\documentclass[11pt,reqno]{amsart}

\usepackage{geometry}
\usepackage{amsmath}
\usepackage{amssymb}
\usepackage{amsfonts}
\usepackage{mathrsfs}
\usepackage{commath}
\usepackage{todonotes}
\usepackage{centernot}
\usepackage{enumerate}
\usepackage{comment}
\usepackage{mdframed}

\usepackage{graphicx}
\usepackage{caption}
\usepackage{subcaption}
\usepackage{pgf,tikz,tikz-cd}
\usepackage{mathrsfs}
\usetikzlibrary{arrows}
\usetikzlibrary{patterns}
\usepackage{verbatim}

\usepackage[backend=bibtex,style=authoryear-comp,maxcitenames=6,maxbibnames=6,doi=false,url=false,isbn=false]{biblatex}
\bibliography{dec}

\usepackage{amsthm}

\newtheorem{theorem}{Theorem}[section]
\newtheorem{proposition}[theorem]{Proposition}
\newtheorem{lemma}[theorem]{Lemma}
\newtheorem{corollary}[theorem]{Corollary}

\theoremstyle{definition}
\newtheorem{example}[theorem]{Example}
\newtheorem{definition}[theorem]{Definition}

\usepackage{color}
\definecolor{darkblue}{rgb}{0.0,0.0,0.4}
\usepackage[pdfauthor={Erick Schulz, Gantumur Tsogtgerel},pdftitle={Convergence of discrete exterior calculus},pdfsubject={convergence, consistency, stability, discrete exterior calculus},colorlinks,citecolor=darkblue,linkcolor=darkblue,urlcolor=darkblue]{hyperref}
\usepackage{hyperref}

\newcommand{\R}{\mathbb{R}}

\newcommand{\vol}{\mathrm{vol}}
\newcommand{\diam}{\mathrm{diam}}

\newcommand{\sgn}{\mathrm{sign}}
\newcommand{\hodge}{\star}
\newcommand{\dual}{\ast}

\mdfsetup{skipabove=+12pt, skipbelow=+12pt}

\title[Convergence of discrete exterior calculus]{Convergence of discrete exterior calculus approximations for Poisson problems}
\author{Erick Schulz}
\address{Department of Mathematics and Statistics.
McGill University,
Montreal, Quebec H3A 0B9, Canada}
\email{erick.schulz@mail.mcgill.ca}

\author{Gantumur Tsogtgerel}
\address{Department of Mathematics and Statistics.
McGill University,
Montreal, Quebec H3A 0B9, Canada}
\email{gantumur@math.mcgill.ca}
\urladdr{http://www.math.mcgill.ca/gantumur/}

\date{\today}                                           

	\definecolor{ffqqqq}{rgb}{1.,0.,0.}
	\definecolor{qqqqff}{rgb}{0.,0.,1.}

\begin{document}

\begin{abstract}
Discrete exterior calculus (DEC) is a framework for constructing discrete versions of exterior differential calculus objects, and is widely used in
computer graphics, computational topology,
and discretizations of the Hodge-Laplace operator and other related partial differential equations.
However, a rigorous convergence analysis of DEC has always been lacking;
as far as we are aware,
the only convergence proof of DEC so far appeared is for the scalar Poisson problem in two dimensions,
and it is based on reinterpreting the discretization as a finite element method.
Moreover, even in two dimensions, there have been some puzzling numerical experiments reported in the literature, apparently suggesting that there is convergence without consistency.

In this paper, we develop a general independent framework for analyzing issues such as convergence of DEC without relying on theories of other discretization methods,
and demonstrate its usefulness by establishing convergence results for DEC beyond the Poisson problem in two dimensions.
Namely, we prove that DEC solutions to the scalar Poisson problem in arbitrary dimensions converge pointwise to the exact solution at least linearly with respect to the mesh size.
We illustrate the findings by various numerical experiments, which show that the convergence is in fact of second order when the solution is sufficiently regular.
The problems of explaining the second order convergence, and of proving convergence for general $p$-forms remain open.
\end{abstract}



\maketitle


\section{Introduction}
\label{s:intro}

The main objective of this paper is to establish convergence of discrete exterior calculus
approximations on unstructured triangulations for the scalar Poisson problem in general dimensions.
There are several approaches to extending the exterior calculus to discrete spaces;
What we mean by {\em discrete exterior calculus} (DEC) in this paper is the approach put forward by Anil Hirani and his collaborators, cf. \parencite{Hirani03,Desbrun2005}.
Since its conception, DEC has found many applications and has been extended in various directions, including
general relativity \parencite{Frau06},
electrodynamics \parencite{Stern2007},
linear elasticity \parencite{Yavari08},
computational modeling \parencite{Desbrun2008},
port-Hamiltonian systems \parencite{SSS2012}, digital geometry processing \parencite{Crane2013},
Darcy flow \parencite{HNC15}, and the Navier-Stokes equations \parencite{MHS16}.
However, a rigorous convergence analysis of DEC has always been lacking;
As far as we are aware,
the only convergence proof of DEC so far appeared is for the scalar Poisson problem in two dimensions,
and it is based on reinterpreting the discretization as a finite element method, cf. \parencite{HPW06}.
In the current paper, we develop a general independent framework for analyzing issues such as convergence of DEC without relying on theories established for other discretization methods,
and demonstrate its usefulness by proving that DEC solutions to the scalar Poisson problem in arbitrary dimensions converge
to the exact solution as the mesh size tends to $0$. 

Developing an original framework to study the convergence of DEC allows us to explore to what extent the theory is compatible in the sense of \parencite{arnold2006}. At the turn of the millennium, compatibility appeared as a conducive paradigm for stability. In this spirit, we here reproduce at the discrete level a standard variational technique in the analysis of PDEs based on the use of the Poincar\'e inequality.

In what follows, we would like to describe our results in more detail.
Let $K_h$ be a family of $n$-dimensional completely well-centered simplicial complexes triangulating a $n$-dimensional polytope $\mathcal{P}$ in $\R^n$,
with the parameter $h>0$ representing the diameter of a typical simplex in the complex.
Let $\Delta_h=\delta_h\dif_{\,h}+\dif_{\,h}\delta_h$ be their associated discrete Hodge-Laplace operators, where the discrete exterior derivatives $\dif_{\,h}$ and the codifferentials $\delta_h=(-1)^{n(k-1)+1}\hodge_h\dif_{\,h}\hodge_h$ are defined as in \parencite{Desbrun2005,Hirani03} up to a sign. Denote by $C^k(K_h)$ the space of $k$-cochains over $K_h$. In this paper, we study the convergence of solutions $\omega_h\in C^k(K_h)$ solving discrete Hodge-Laplace Dirichlet problems of the form
\begin{equation}\label{eq: Discrete Poisson Dirichlet}
\begin{cases}
\Delta_h \omega_h = R_h f& \mbox{in } K_h\\
\omega_h=R_h g& \mbox{on } \partial K_h,
\end{cases}
\end{equation}
where $f$ and $g$ are differential forms, and $R_h$ is the deRham operator,
cf. \parencite{Whitney57}. It is shown in \parencite{Xu04,Xu04b} that under {\em very} special symmetry assumptions on $K_h$, the consistency estimate
\begin{equation*}
\|\Delta_h R_h \omega - R_h\Delta \omega\|_{\infty}= O(h^2)
\end{equation*}
holds for sufficiently smooth functions $\omega$.
However, as the numerical experiments from \parencite{Xu04,Nong04} revealed, this does \textit{not} hold for general triangulations.
In Section~\ref{s:consistency}, we show that a common shape regularity assumption on $K_h$ can only guarantee
\begin{equation}\label{eq: consistency estimate}
\|\Delta_h R_h \omega - R_h\Delta \omega\| = O(1)+O(h)
\end{equation}
for sufficiently smooth functions $\omega$, in both the maximum $\|\cdot\|_\infty$ and discrete $L^2$-norm $\|\cdot\|_h$.

Although the consistency estimate \eqref{eq: consistency estimate} is not adequate for the Lax-Richtmyer framework,
by making use of a special structure of the error term itself, we are able to establish convergence for $0$-forms in general dimensions.
Namely, we prove in Section \ref{s:stability} that the approximations $u_h\in C^0(K_h)$ obtained from solving (\ref{eq: Discrete Poisson Dirichlet}) still satisfy
\begin{equation}\label{intro convergence}
\|\omega_h-R_h \omega\|_h = O(h) ,
\qquad\text{and}\qquad
\|\dif\,(\omega_h-R_h \omega)\|_h=O(h) ,
\end{equation}
where $\omega\in C^2(\mathcal{P})$ is the solution of the corresponding continuous Poisson problem with source term $f$ and boundary condition $g$.
We remark that a convergence proof in 2 dimensions can be obtained by reinterpreting the discrete problem as arising from affine finite elements,
cf. \parencite{HPW06,Wardetzky08}, which shows that
the first quantity in \eqref{intro convergence} should indeed be $O(h^2)$ if the exact solution is sufficiently regular.
This is consistent with the numerical experiments from \parencite{Nong04},
and moreover, in Section \ref{s:numerics}, we report some numerical experiments in 3 dimensions, which suggest that one has $O(h^2)$ convergence in general dimensions.
Therefore, our $O(h)$ convergence result for $0$-forms should be considered only as a first step towards a complete theoretical understanding of the convergence behavior of discrete exterior calculus.
Apart from explaining the $O(h^2)$ convergence for $0$-forms, proving convergence for general $p$-forms remains as an important open problem.

This paper is organized as follows.
In the next section, we review the basic notions of discrete exterior calculus, not only to fix notations,
but also to discuss issues such as the boundary of a dual cell in detail to clarify some inconsistencies existing in the current literature.
Then in Section~\ref{s:consistency}, we treat the consistency question, and in Section~\ref{s:stability}, we establish stability of DEC for the scalar Poisson problem.
Our main result \eqref{intro convergence} is proved in Section~\ref{s:stability}.
We end the paper by reporting on some numerical experiments, in Section \ref{s:numerics}.

\section{Discrete environment}

We review the basic definitions involved in discrete exterior calculus that are needed for the purposes of this work. Readers unaquainted with these discrete structures are encouraged to go over the material covered in the important works cited below. The shape regularity condition that we impose on our triangulations is discussed in Subsection \ref{subs: Discrete Structures}, and a new definition for the boundary of a dual cell is introduced in Subsection \ref{subs: Discrete Operators}.

\subsection{Simplicial complexes and regular triangulations}
\label{subs: Discrete Structures}

The basic geometric objects upon which DEC is designed are borrowed from algebraic topology. While the use  of cube complexes is discussed in \parencite{BH12}, we here consider simplices to be the main building blocks of the theory. The following definitions are given in \parencite{Desbrun2005} and can also be found in any introductory textbook on simplicial homology.

By a \textit{$k$-simplex} in $\mathbb{R}^n$, we mean the $k$-dimensional convex span
	\begin{equation*}
		\sigma=\left.\lbrack v_0,...,v_k\rbrack=\left\{\sum_{i=0}^k a_i v_i\right\vert v_i\in\mathbb{R}^n, a_i\geq0,\sum_{i=0}^k a_i =1\right\},
	\end{equation*}
where $v_0,...,v_k$ are affinely independent. We denote its circumcenter by $c(\sigma)$.
Any $\ell$-simplex $\tau$ defined from a nonempty subset of these vertices is said to be a face of $\sigma$, and we denote this partial ordering by $\sigma\succ\tau$. We write $\vert\sigma\vert$ for the $k$-dimensional volume of a $k$-simplex $\sigma$ and adopt the convention that the volume of a vertex is $1$. The plane of $\tau$ is defined as the smallest affine subspace of $\mathbb{R}^n$ containing $\tau$ and is denoted by $P(\tau)$.
A simplicial \textit{$n$-complex} $K$ in $\mathbb{R}^n$ is a collection of $n$-simplices in $\mathbb{R}^n$ such that:
\begin{itemize}
		\item Every face of a simplex in $K$ is in $K$;
		\item The intersection of any two simplices of $K$ is either empty or a face of both.
\end{itemize}
It is well-centered if the the circumcenter of a simplex of any dimension is contained in the interior of that simplex. If the set of simplices of a complex $L$ is a subset of the simplices of another complex $K$, then $L$ is called a \textit{subcomplex} of $K$. We denote by $\Delta_k(L)$ the set of all elementary $k$-simplices of $L$. The \textit{star} of a $k$-simplex $\sigma \in K$ is defined as the set $\text{St}(\sigma)=\{\rho\in K \vert \sigma \prec \rho\}$. $\text{St}(\sigma)$ is not closed under taking faces in general. It is thus useful to define the closed star $\overline{\text{St}}(\sigma)$ to be the smallest subcomplex of $K$ containing $\text{St}(\sigma)$. We denote the free abelian group generated by a basis consisting of the oriented k-simplices by $C_k(K)$. This is the space of finite formal sums of $k$-simplices with coefficients in $\mathbb{Z}$. Elements of $C_k(K)$ are called \textit{$k$-chains}. More generally, we will write $\oplus_kC_k(K)$ for the space of finite formal sums of elements in $\cup_k C_k(K)$ with coefficients in $\mathbb{F}_2$. The \textit{boundary} of a $k$-chain is obtained using the linear operator $\partial:\oplus_k C_k(K)\longrightarrow \oplus_k C_{k-1}(K)$ defined on a simplex $[v_0,...,v_k]\in K$ by
\begin{equation*}
\partial\lbrack v_0,...,v_k\rbrack = \sum_{i=0}^k (-1)^i \lbrack v_0,...,\hat{v_i},...,v_k\rbrack.
\end{equation*}

 Any simplicial $n$-complex $K$ in $\mathbb{R}^n$ such that $\cup_{\sigma\in\Delta_n(K)}\sigma=\mathcal{P}$ is called a \textit{triangulation} of $\mathcal{P}$. We consider in this work families of well-centered triangulations $K_h$ in which each complex is indexed by the size $h>0$ of its longest edge. We write $\gamma_{\tau}$ for the radius of the largest $\dim(\tau)$-ball contained in $\tau$. The following condition imposed on $K_h$ is common in the finite element literature, see, e.g., \parencite{Ciarlet02}. We suppose there exists a shape regularity constant $C_{reg}>0$, independent of $h$, such that for all simplex $\sigma$ in $K_h$,
\begin{equation*}
\diam(\sigma)\leq C_{\text{reg}}\,\gamma_{\sigma}.
\end{equation*}
A family of triangulations satisfying this condition is said to be \textit{regular}. Important consequences of this regularity assumption will later follow from the next lemma.
\begin{lemma}\label{lem: bounded number of simplices}
Let $\{K_h\}$ be a regular family of simplicial complexes triangulating an $n$-dimensional polytope $\mathcal{P}$ in $\mathbb{R}^n$.  Then there exists a positive constant $C_{\#}$, independent of $h$, such that
\begin{equation*}
\#\left\{\sigma\in\Delta_n(K_h) : \sigma\in \overline{\mathrm{St}}(\tau)\right\}\leq C_{\#},\,\,\,\,\,\forall\,\tau\in\Delta_k(K_h).
\end{equation*}
\end{lemma}
\begin{proof}
If $\tau\prec\sigma$, then $\sigma$ contains all the vertices of $\tau$. It is thus sufficient to consider the case where $\tau$ is a vertex. Let $v$ be a vertex in $K_h$, and suppose $\sigma\in\overline{\text{St}}(v)$. Let $B_\sigma$ be the $n$-dimensional ball of radius $\gamma_\sigma$ contained in $\sigma$. The set $A=\cup_{\eta\in\Delta_{k-1}(\sigma)}\eta\cap B_{\sigma}\cap\text{St}(\sigma)$ contains exactly $n$ points in $\mathbb{R}^n$. The argument
\begin{equation*}
\theta=\min_{A}\arctan \left(\frac{\gamma_{\sigma}}{\vert x - v\vert}\right)
\end{equation*}
is thus well-defined, and it satisfies $\theta\geq \arctan\left(1/C_{\text{reg}}\right)$ by the regularity assumption.  Consider an $n$-dimensional cone of height $\gamma_{\sigma}$, apex $v$ and aperture $2\theta$ contained in $\sigma$. Its generatrix determine a spherical cap $V_\sigma$ on the hypersphere $S_{\sigma}$ of radius $h$ centered at $v$. The intersection of spherical caps determined by  cones contained in distinct simplices is countable. Therefore, there can be at most
\begin{equation*}
\frac{\vol (S_{\sigma})}{V(\arctan\left(1/C_{\text{reg}}\right))} = 2\pi\left(\int_0^{\arctan(1/C_{\text{reg}})}\sin^n(t)dt\right)^{-1}
\end{equation*}
of them, and only so many distinct $n$-simplices containing $v$ as a result.
\end{proof}

\subsection{Combinatorial orientation and duality}\label{sec: Combinatorial Orientation and Duality}
A triangulation stands as a geometric discretization of a polytopic domain, but it ought to be equipped with a meaningful notion of orientation if a compatible discrete calculus is to be defined on its simplices. We outline below the exposition found in \parencite{Hirani03}.

We felt the need to review the definitions of Subsection \ref{subs: Discrete Structures} partly to be able to stress the fact that the expression of a $k$-simplex  $\sigma$ comes naturally with an ordering of its vertices. Defining two orderings to be equivalent if they differ by an even permutation yields two equivalence classes called \textit{orientations}. The vertices themselves are dimensionless. As such, they are given no orientation. By interpreting a permutation $\rho$ of the vertices in $\sigma$ as an ordering for the basis vectors $v_{\rho(1)}-v_{\rho(0)},v_{\rho(2)}-v_{\rho(1)},...,v_{\rho(k)}-v_{\rho(k-1)}$, we see that these equivalence classes coincide with the ones obtained when the affine space $P(\sigma)$ is endowed with an orientation in the usual sense. A simplex is thus oriented by its plane and vice versa. The planes of the $(k-1)$-faces of $\sigma$ inherits an orientation as subspaces of $P(\sigma)$. Correspondingly, we establish that the \textit{induced orientation} by $[v_0,...,v_k]$ on the $(k-1)$-face $v_0,...,\hat{v_i},...,v_k$ is the same as the orientation of $[v_0,...,\hat{v_i},...,v_k]$ if $i$ is even, while the opposite orientation is assigned otherwise.
Two $k$-simplices $\sigma$ and $\tau$ in $\mathbb{R}^n$ are hence comparable if $P(\sigma)=P(\tau)$, or if they share a ($k-1$)-face. In the first case, the \textit{relative orientation} $\text{sign}(\sigma,\tau)=\pm 1$ of the two simplices is $+1$ when their bases yield the same orientation of the plane and $-1$ otherwise. The induced orientation on the common face is similarly used to establish relative orientation in the second case. The mechanics of orientation are conveniently captured by the structure of exterior algebra. For example, $v_1-v_0\wedge ... \wedge v_k-v_{k-1}$ could be used to represent the orientation of $\lbrack v_0,...,v_k\rbrack$. From now on, we will always assume that $K_h$ is well-centered and that all its $n$-simplices  are positively oriented with respect to each other. The orientations of lower dimensional simplices are chosen independently.

We are now ready to introduce the final important objects pertaining to the discrete domain before we move on to the functional aspects of DEC. Denote by $D_h$ the smallest simplicial $n$-complex in $\mathbb{R}^n$ containing every simplex of the form $[c(v),...,c(\tau),...,c(\sigma)]$, where the simplices $v\prec ...\prec\tau\prec ...\prec \sigma$ belong to $K_h$. Let $\dual: \oplus_kC_k(K_h)\longrightarrow \oplus_kC_{n-k}(D_h)$ be the homomorphism acting on $\tau\in K_h$ by
\begin{equation*}\label{def: circumcentric duality operator}
\dual\tau = \sum_{\tau\prec...\prec\sigma}\pm_{\sigma,...,\tau}[c(\tau),...,c(\sigma)],
\end{equation*}
where $\pm_{\sigma,...,\tau}=\sgn\left(\lbrack v,...,c\left(\tau\right)], \tau\right)
\sgn\left(\lbrack v,...,c\left(\sigma\right)\rbrack, \sigma\right)$. We define the \textit{oriented circumcentric dual} of $K_h$ by $\dual K_h=\{\dual\tau \,\vert\, \tau\in K_h\}$ (see Figure \ref{fig: example of dual}). Note that the space $C_{n-k}(\dual K_h)$ of finite formal sums of ($n-k$)-dimensional elements of $\dual K_h$ with integer coefficients is a subgroup of $C_{n-k}(D_h)$ onto which the restriction of $\dual$ to $C_k(K_h)$ is an isomorphism. In fact, since a simplex whose vertices are $c(\tau),...,c(\sigma)$ is oriented in the above formula to satisfy $ P(\tau)\times P(c(\tau),...,c(\sigma))\sim \sigma$, we effortlessly extend the definition of $\dual$ to $\oplus_k C_k(\dual K_h)$ as well by defining $\dual\dual\tau=(-1)^{k(n-k)}\tau$ for every simplex $\tau\in\Delta_k(K_h)$, $k=1,...,n$.
\begin{figure}
	\centering
	\begin{subfigure}[b]{0.45\textwidth}
		\centering
		\begin{tikzpicture}[scale = 0.7,line cap=round,line join=round,>=triangle 45,x=1.0cm,y=1.0cm]
		\draw [dash pattern=on 5pt off 5pt,color=qqqqff] (-1.,0.)-- (3.02,-2.62);
		\draw [dash pattern=on 5pt off 5pt,color=qqqqff] (3.02,-2.62)-- (7.,0.);
		\draw [line width=2.4pt,color=qqqqff] (7.,0.)-- (3.04,2.5);
		\draw [line width=2.4pt,color=qqqqff] (3.04,2.5)-- (3.02,-2.62);
		\draw [line width=2.4pt,color=ffqqqq] (1.48,-0.02)-- (4.5,0.);
		\draw [dash pattern=on 5pt off 5pt,color=qqqqff] (3.04,2.5)-- (-1.,0.);
		\draw [line width=2.4pt,color=ffqqqq] (4.5,0.)-- (5.212442320669719,1.1285086359408343);
		\draw (3.1036881946132455,-1.4) node[anchor=north west] {$\sigma_1$};
		\draw (1.8306602313282923,-0.22120798958725374) node[anchor=north west] {$\star \sigma_1$};
		\draw (6.107238545488682,1.0916020975503622) node[anchor=north west] {$\sigma_2$};
		\draw (5.0,0.5512752833293447) node[anchor=north west] {$\star\sigma_2$};
		\draw (3.031566901139814,0.3411266917924226)-- (3.32,0.34);
		\draw (3.32,0.34)-- (3.3223033067274796,-0.0077993158494868865);
		\draw (0.9,1.52)-- (0.6303039217423666,1.3531088624643357);
		\draw (0.9,1.52)-- (1.0540431338496163,1.27106629569902);
		\draw [dash pattern=on 5pt off 5pt,color=ffqqqq] (1.48,-0.02)-- (-0.18604292467069117,-2.5762948691512126);
		\draw [dash pattern=on 5pt off 5pt,color=ffqqqq] (4.5,0.)-- (6.119807658452332,-2.4606238475726263);
		\draw [dash pattern=on 5pt off 5pt,color=ffqqqq] (5.212442320669718,1.1285086359408345)-- (6.501264118952639,3.1700023644209794);
		\draw [dash pattern=on 5pt off 5pt,color=ffqqqq] (0.7843470555919828,1.1041751581633557)-- (1.48,-0.02);
		\draw [dash pattern=on 5pt off 5pt,color=ffqqqq] (0.7843470555919828,1.1041751581633557)-- (-0.496594014052013,3.2000546617410786);
		\begin{scriptsize}
		\draw [fill=qqqqff] (-1.,0.) circle (1.0pt);
		\draw [fill=qqqqff] (3.04,2.5) circle (2.5pt);
		\draw [fill=qqqqff] (3.02,-2.62) circle (2.5pt);
		\draw [fill=qqqqff] (7.,0.) circle (2.5pt);
		\draw [fill=ffqqqq] (1.48,-0.02) circle (2.5pt);
		\draw [fill=ffqqqq] (4.5,0.) circle (2.5pt);
		\draw [fill=ffqqqq] (5.212442320669718,1.1285086359408345) circle (2.5pt);
		\draw [fill=ffqqqq] (4.5,0.) circle (2.5pt);
		\draw [fill=ffqqqq] (0.7843470555919828,1.1041751581633557) circle (2.5pt);
		\end{scriptsize}
		\end{tikzpicture}
	\end{subfigure}
	\begin{subfigure}[b]{0.45\textwidth}
		\centering
		\begin{tikzpicture}[scale=0.73,line cap=round,line join=round,>=triangle 45,x=1.0cm,y=1.0cm]
		(14.359768562151523,7.709972748929639);
		\fill[fill=black,fill opacity=0.55] (3.04,2.5) -- (0.7843470555919828,1.1041751581633557) -- (1.48,-0.02) -- cycle;
		\fill[fill=black,fill opacity=0.15] (3.04,2.5) -- (5.212442320669718,1.1285086359408345) -- (4.5,0.) -- cycle;
		\fill[fill=black,fill opacity=0.4] (3.04,2.5) -- (1.48,-0.02) -- (3.030196352347691,-0.009733798991074892) -- cycle;
		\fill[fill=black,fill opacity=0.3] (3.04,2.5) -- (3.030196352347691,-0.009733798991074892) -- (4.5,0.) -- cycle;
		\draw [dash pattern=on 7pt off 7pt] (7.,0.)-- (3.04,2.5);
		\draw [dash pattern=on 7pt off 7pt] (3.04,2.5)-- (-1.,0.);
		\draw (3.1915055989405454,3.152042040793411) node[anchor=north west] {$\sigma$};
		\draw (3.04,2.5)-- (0.7843470555919828,1.1041751581633557);
		\draw (1.48,-0.02)-- (3.04,2.5);
		\draw (3.04,2.5)-- (5.212442320669718,1.1285086359408345);
		\draw (4.5,0.)-- (3.04,2.5);
		\draw [dash pattern=on 7pt off 7pt] (3.02,-2.62)-- (3.030196352347691,-0.009733798991074892);
		\draw (3.04,2.5)-- (3.030196352347691,-0.009733798991074892);
		\draw (3.04,2.5)-- (1.48,-0.02);
		\draw (3.030196352347691,-0.009733798991074892)-- (3.04,2.5);
		\draw (3.04,2.5)-- (3.030196352347691,-0.009733798991074892);
		\draw (4.5,0.)-- (3.04,2.5);
		\draw [line width=2.4pt,color=ffqqqq] (4.5,0.)-- (1.48,-0.02);
		\draw [line width=2.4pt,color=ffqqqq] (0.7843470555919828,1.1041751581633557)-- (1.48,-0.02);
		\draw [line width=2.4pt,color=ffqqqq] (4.5,0.)-- (5.212442320669718,1.1285086359408345);
		\draw (4.00,1.4461498459354074) node[anchor=north west] {$\star\sigma$};
		\draw [line width=2.8pt,color=ffqqqq] (0.7843470555919828,1.1041751581633557)-- (3.04,2.5);
		\draw [line width=2.8pt,color=ffqqqq] (3.04,2.5)-- (5.212442320669718,1.1285086359408345);
		\begin{scriptsize}
		\draw [fill=ffqqqq] (3.04,2.5) circle (2.5pt);
		\draw [fill=ffqqqq] (1.48,-0.02) circle (2.5pt);
		\draw [fill=ffqqqq] (4.5,0.) circle (2.5pt);
		\draw [fill=ffqqqq] (0.7843470555919828,1.1041751581633557) circle (2.5pt);
		\draw [fill=ffqqqq] (5.212442320669718,1.1285086359408345) circle (2.5pt);
		\draw [fill=ffqqqq] (3.0300527348996282,-0.04649986569513942) circle (2.5pt);
		\end{scriptsize}
		\end{tikzpicture}
	\end{subfigure}
	\caption{Boundary and interior faces are perpendicular to their duals. On the left is shown a $1$-dimensional complex in blue and its dual in red. On the right, each shade of grey indicates a $2$-dimensional simplex in $D_h$. The dual of the vertex $\sigma$ is a complex of these simplices, and its boundary is colored in red.}\label{fig: example of dual}
\end{figure}
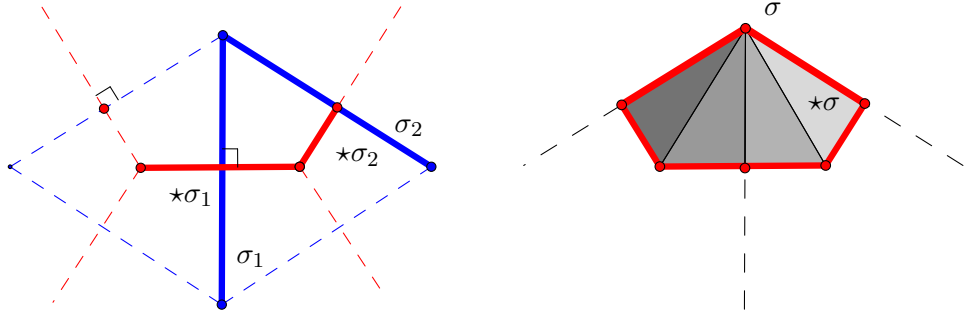\

\subsection{Discrete operators}\label{subs: Discrete Operators}
 A strength of discrete exterior calculus is foreseen in the following objects. The theory is built on a natural and intuitive notion of discrete differential forms having a straightforward implementation.  We call \textit{$k$-cochains} the elements of the space $C^k(K_h) = \text{Hom}(C_k(K_h), \mathbb{R})$, and define $\oplus_kC^k(K_h)$, $C^k(\dual K_h)$ and $\oplus_kC^k(\dual K_h)$ as expected. The \textit{discrete Hodge star} is an isomorphism $\hodge_h:\oplus_k C^{k}(K_h)\longrightarrow \oplus_kC^{n-k}(\dual K_h)$ satisfying
\begin{equation*}\label{def: Hodge star}
\langle \hodge_h\omega_h,\dual\sigma\rangle = \frac{\vert\dual\sigma\vert}{\vert\sigma\vert}\langle \omega_h,\sigma\rangle,\,\,\,\,\,\forall\,\omega\in K_h.
\end{equation*}
In complete formal analogy, we further impose that
\begin{equation*}
\langle \hodge_h\hodge_h\omega_h,\dual\dual\sigma\rangle = \left(\frac{\vert\dual\sigma\vert}{\vert\sigma\vert}\right)^{-1}\langle \hodge_h\omega_h,\dual\sigma\rangle,\,\,\,\,\,\forall\,\omega\in K_h.
\end{equation*}
In other words, $\langle \hodge_h\hodge_h\omega_h,\sigma\rangle=(-1)^{k(n-k)}\langle \hodge_h\hodge_h\omega_h,\dual\dual\sigma\rangle=(-1)^{k(n-k)}\langle \omega_h,\sigma\rangle$ for all $k$-cochains $\omega_h$, and as a consequence $\hodge_h\hodge_h=(-1)^{k(n-k)}$ on $C^k(K_h)$. The spaces $C^k(K_h)$ are finite dimensional Hilbert spaces when equipped with the discrete inner product
\begin{equation*}
\left(\alpha_h,\beta_h\right)_h=\sum_{\tau\in\Delta_k(K_h)}\langle\alpha_h,\tau\rangle\langle\hodge_h\beta_h,\dual\tau\rangle,\,\,\,\,\,\alpha_h,\beta_h\in C^k(K_h).
\end{equation*}
A compatible definition of the discrete $L^2$-norm on the dual triangulation is also chosen by enforcing the Hodge star to be an isometry. The norm in the following definition is obtained from the inner product
\begin{equation*}
\left(\hodge\alpha_h,\hodge\beta_h\right)_h=\sum_{\tau\in\Delta_k(K_h)}\langle\hodge\alpha_h,\dual\tau\rangle\langle\hodge_h\hodge_h\beta_h,\dual\dual\tau\rangle,\,\,\,\,\,\hodge\alpha_h,\hodge\beta_h\in C^k(\dual K_h),
\end{equation*}
which is immediatly seen to satisfy $\left(\alpha_h,\beta_h\right)_h=\left(\hodge\alpha_h,\hodge\beta_h\right)_h$.
\begin{definition}
	The discrete $L^2$-norm on $C^k(\dual K_h)$ is defined by
	\begin{equation*}
	\|\hodge\omega_h\|^2_h = \sum_{\tau\in\Delta_k(K_h)}\left(\frac{\vert\dual\tau\vert}{\vert\tau\vert}\right)^{-1}\langle\hodge_h\omega_h,\dual\tau\rangle^2,\,\,\,\,\, \omega_h\in C^k(\dual K_h).
	\end{equation*}
\end{definition}

The last step towards a full discretization of the Hodge-Laplace operator is to discretize the exterior derivative. While requiring that
\begin{equation*}
 \langle \dif\omega_h,\tau\rangle =\langle \omega_h,\partial\tau\rangle,\,\,\,\,\,\forall\,\tau\in K_h,
\end{equation*}
readily defines a \textit{discrete exterior derivative} $\dif_{\,h}:\oplus_kC^k(K_h)\longrightarrow \oplus_kC^{k+1}(K_h)$ that satisfies Stokes' theorem, we need to make precise what is meant by the boundary of a dual element if we hope to define the exterior derivative on $C^k(\dual K_h)$ similarly. We examine the case of a $k$-simplex $\tau=\pm_{\tau}[v_0,...,v_k]$ in $K_h$. It is sufficient to restrict our attention to an $n$-dimensional simplex $\sigma$ to which $\tau$ is a face. We assume that the orientation of $\sigma$ is consistent with
$
\left(v_1-v_0\right)\wedge\left(v_2-v_1\right)\wedge...\wedge\left(v_n-v_{n-1}\right),
$
and thereon begin our study. The orientation of $\tau$ is represented by
$
\pm_{\tau}\left(v_1-v_0\right)\wedge...\wedge\left(v_k-v_{k-1}\right)
$. We have seen that enforcing the orientation of its circumcentric dual
$
\dual\tau=\pm_{\dual\tau}[c(\tau),c([v_0,...,v_{k+1}]),...,c(\sigma)]
$
to satisfy the definition given in Section \ref{sec: Combinatorial Orientation and Duality} is equivalent to requiring that
$
\pm_{\tau}\pm_{\dual \tau}\left(v_1-v_0\right)\wedge...\wedge\left(v_n-v_{n-1}\right) \sim \sigma
$.
Consequently,  $\pm_{\tau}=\pm_{\dual \tau}$ by hypothesis (i.e. the signs agree), and we obtain
$
\dual\tau\sim \pm_{\tau}\left(v_{k+1}-v_k\right)\wedge...\wedge\left(v_n-v_{n-1}\right)
$.
Since similar equivalences hold for any  $(k+1)$-face $\eta=\pm_{\eta}[v_0,...,v_{k+1}]$ of $\sigma$, the orientation induced by $\eta$ on the face $v_0,...,v_k$ is defined to be the same as the orientation of
\begin{equation*}
(-1)^{k+1}\pm_{\eta}\pm_{\tau}\tau
\sim(-1)^{k+1}\pm_{\eta}\left(v_1-v_0\right)\wedge...\wedge\left(v_{k}-v_{k-1}\right).
\end{equation*}
Therefore, choosing $\pm_{\text{old}}=(-1)^{k+1}\pm_{\eta}\pm_{\tau}$ makes the induced orientation of $\pm_{\text{old}}\eta$ on that face consistent with the orientation of $\tau$, and yields on the one hand that the orientation of $c(\eta)$, ..., $c(\sigma)$ as a piece of $\dual\pm_{\text{old}}\eta$ is given by
\begin{equation*}
(-1)^{k+1}\pm_{\tau}(\pm_{\eta})^2\left(v_{k+2}-v_{k+1}\right)\wedge...\wedge\left(v_n-v_{n-1}\right).
\end{equation*}
 On the other hand, $P\left(\pm_{\text{new}}\left(\dual\pm_{\text{old}}\eta\right)\right)$ is a subspace of codimension $1$ in $P(\dual\tau)$. As such, it inherits a boundary orientation in the usual sense through the assignement
\begin{equation*}
P(\dual\tau)\sim\pm_{\text{new}}(-1)^{k+1}\pm_{\tau}\overrightarrow{\nu}\wedge\left(v_{k+2}-v_{k+1}\right)\wedge...\wedge\left(v_n-v_{n-1}\right),
\end{equation*}
where $\overrightarrow{\nu}$ is any vector pointing away from $P(\dual\eta)$. Since by hypothesis $\overrightarrow{\nu}=(v_{k+1}-v_k)$ is a valid choice of outward vector, we conclude that the above compatible relation
holds if and only if $\pm_{\text{new}}=(-1)^{k+1}$. The claim that the following new definition for the boundary of a dual element is suited for integration by parts rests on the later observation.

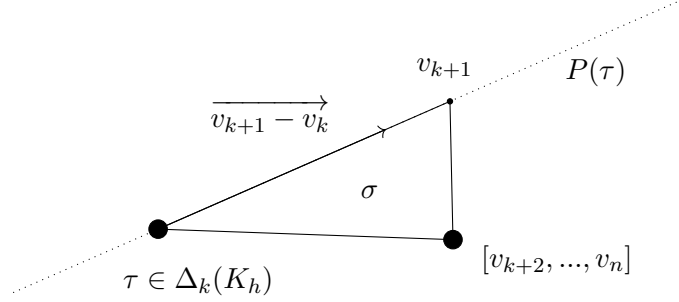
\begin{figure}
	\centering
	\begin{tikzpicture}
	\draw (2.12,1.14)-- (6.,2.84);
	\draw (1.52,0.78) node[anchor=north west] {$\tau\in\Delta_k(K_h)$};
	\draw (5.42,3.54) node[anchor=north west] {$v_{k+1}$};
	\draw (2.68,3.00) node[anchor=north west] {$\overrightarrow{v_{k+1}-v_k}$};
	\draw [dotted] (0.19310314081273283,0.2957410668509395)-- (9.06181048126435,4.181514901584896);
	\draw (2.12,1.14)-- (6.04,1.);
	\draw (6.04,1.)-- (6.,2.84);
	\draw (6.28,1.06) node[anchor=north west] {$[v_{k+2},...,v_n]$};
	\draw (4.68,1.86) node[anchor=north west] {$\sigma$};
	\draw (7.4,3.56) node[anchor=north west] {$P(\tau)$};
	\draw [->,line width=0.4pt] (2.12,1.14) -- (5.148252602483226,2.4668117072735782);
	\begin{scriptsize}
	\draw [fill=black] (2.12,1.14) circle (3.5pt);
	\draw [fill=black] (6.,2.84) circle (1.0pt);
	\draw [fill=black] (6.04,1.) circle (3.5pt);
	\end{scriptsize}
	\end{tikzpicture}
	\caption{This figure illustates the exposition of Section \ref{subs: Discrete Operators}. The plane in which the triangle lies is a representation for the $n$-dimensional plane of $\sigma$, while the dotted line represents $P(\tau)$.}
\end{figure}

\begin{definition}\label{def: boundary of dual}
The linear operator $\partial: \oplus_kC_{n-k}\left(\dual K_h,\mathbb{Z}\right)\longrightarrow \oplus C_{n-k-1}\left(\dual K_h,\mathbb{Z}\right)$, which we call the \textit{dual boundary operator}, is defined as the linear operator acting on the dual of $\tau=[v_0,...,v_k]\in K_h$ by
\begin{equation*}
\partial\dual\tau = (-1)^{k+1} \sum_{\eta\succ\tau}\dual\eta,
\end{equation*}
where $\eta$ is a $(k+1)$-simplex in $K_h$ oriented so that the induced orientation of $\eta$ on the face $v_0,...,v_k$ is consistent with the orientation of $\tau$.
\end{definition}

\begin{example}
Let $\sigma=\lbrack v_0,v_1,v_2\rbrack$ be an oriented $2$-simplex. The $2$-dimensional dual of $\tau=v_0$ is given by
\begin{equation*}
\dual \tau = \pm_{v_0,\lbrack v_0,v_1\rbrack,\sigma}\lbrack v_0,c\left(\lbrack v_0,v_1\rbrack\right), c(\sigma) \rbrack +\pm_{v_0,\lbrack v_0,v_2\rbrack,\sigma}\lbrack v_0,c\left(\lbrack v_0,v_2\rbrack\right), c(\sigma) \rbrack,
\end{equation*}
where by definition
\begin{align*}
\pm_{v_0,\lbrack v_0,v_1\rbrack,\sigma} &= \sgn\left(\lbrack v_0,c\left(\lbrack v_0,v_1\rbrack\right), c(\sigma) \rbrack, \sigma\right) = 1;\\
\pm_{v_0,\lbrack v_0,v_2\rbrack,\sigma} &= \sgn\left(\lbrack v_0,c\left(\lbrack v_0,v_2\rbrack\right), c(\sigma) \rbrack, \sigma\right) = -1.
\end{align*}
The orientation of the boundary edges of $\dual\tau$ with endpoints $c([v_0,v_2])$, $c(\sigma)$ and $c([v_0,v_1])$, $c(\sigma)$ are compatible with integration by parts if they are assigned an orientation equivalent to the one given in the above expression. Yet, from the definition of the boundary of a dual cell as found in the literature,
\begin{align*}
\partial\dual\tau &= \dual \left((-1)[v_0,v_1]\right)+\dual\left((-1)[v_0,v_2]\right)\\
& = \dual [v_1,v_0] + \dual [v_2,v_0]\\
&= \pm_{[v_1,v_0],\sigma}[c(v_0,v_2),c(\sigma)] +\pm_{[v_2,v_0],\sigma}[c(v_0,v_1),c(\sigma)],
\end{align*}
where
\begin{align*}
\pm_{[v_1,v_0],\sigma} &= \sgn\left([v_0,c([v_1,v_0])],[v_1,v_0]\right)
\cdot\sgn\left(\lbrack v_0,c\left(\lbrack v_0,v_1\rbrack\right),c(\sigma)\rbrack,\sigma\right)= (-1)(+1)=-1; \\
\pm_{[v_2,v_0]\sigma} &= \sgn\left([v_0,c([v_2,v_0])],[v_2,v_0]\right)
\cdot\sgn\left(\lbrack v_0,c\left(\lbrack v_0,v_2\rbrack\right),c(\sigma)\rbrack,\sigma\right)= (-1)(-1)=+1.
\end{align*}
We illustrate this example in Figure \ref{fig: example 2d dual}.

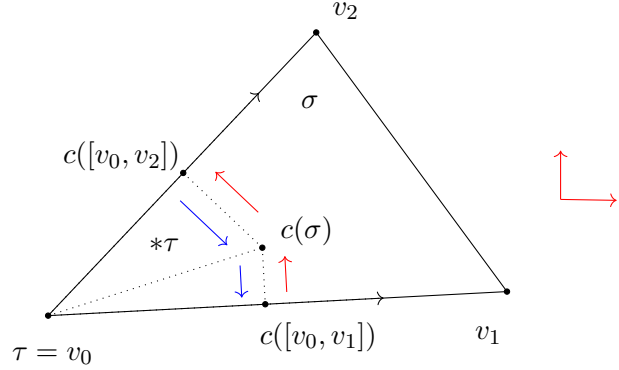
\begin{figure}
\definecolor{ffqqqq}{rgb}{1.,0.,0.}
\definecolor{qqqqff}{rgb}{0.,0.,1.}
\begin{tikzpicture}
\draw (3.14,3.56)-- (-0.4240779728250691,-0.20776608152479198);
\draw (-0.4240779728250691,-0.20776608152479198)-- (5.675922027174931,0.11223391847520807);
\draw (5.675922027174931,0.11223391847520807)-- (3.14,3.56);
\draw [dotted] (-0.4240779728250691,-0.20776608152479198)-- (2.4232725938120594,0.6963661804576204);
\draw (-1.039849439796324,-0.49656573565702594) node[anchor=north west] {$\tau=v_0$};
\draw (2.8062887772578144,2.859378394909821) node[anchor=north west] {$\sigma$};
\draw (0.7826670356051468,0.961447444664151) node[anchor=north west] {$*\tau$};
\draw (3.2210683889009077,4.091148150698402) node[anchor=north west] {$v_2$};
\draw (5.118999339146578,-0.22004599456163035) node[anchor=north west] {$v_1$};
\draw [dotted] (1.371924094155322,1.690878033983527)-- (2.4232725938120594,0.6963661804576204);
\draw [dotted] (2.4232725938120594,0.6963661804576204)-- (2.4627580591631495,-0.05632550279754113);
\draw [->,color=ffqqqq] (2.3690321053383556,1.1670710838847158) -- (1.799605617188187,1.7057139719107774);
\draw [->,color=ffqqqq] (2.7479764647247498,0.1069878641809151) -- (2.722920500440797,0.5846171833437671);
\draw [->,color=qqqqff] (1.3310918011400998,1.32179061725372) -- (1.9576823515357755,0.729073996761772);
\draw [->,color=qqqqff] (2.1275651061286727,0.45934398692386236) -- (2.150025357024407,0.031195454223931884);
\draw (2.5297690361624188,1.2505362649002465) node[anchor=north west] {$c(\sigma)$};
\draw (-0.3611191661985349,2.243493517015531) node[anchor=north west] {$c([v_0,v_2])$};
\draw (2.2532492950670235,-0.15720059885813134) node[anchor=north west] {$c([v_0,v_1])$};
\draw [->,dotted] (-0.4240779728250691,-0.20776608152479198) -- (2.38028849385736,2.7568707569213315);
\draw [->,dotted] (-0.4240779728250691,-0.20776608152479198) -- (4.037006175016316,0.026258004919346323);
\draw [->,color=ffqqqq] (6.388476332357257,1.338519818885145) -- (6.388476332357257,1.9886250851620182);
\draw [->,color=ffqqqq] (6.388476332357257,1.338519818885145) -- (7.134893489934407,1.3259507397444452);
\begin{scriptsize}
\draw [fill=black] (-0.4240779728250691,-0.20776608152479198) circle (1.0pt);
\draw [fill=black] (3.14,3.56) circle (1.0pt);
\draw [fill=black] (5.675922027174931,0.11223391847520807) circle (1.0pt);
\draw [fill=black] (1.371924094155322,1.690878033983527) circle (1.0pt);
\draw [fill=black] (2.4232725938120594,0.6963661804576203) circle (1.0pt);
\draw [fill=black] (2.4627580591631495,-0.05632550279754112) circle (1.0pt);
\end{scriptsize}
\end{tikzpicture}
\caption{In red is shown the orientation of the boundary of $\dual\tau$ as obtained from definition \ref{def: boundary of dual}. The blue vectors represent the orientation as obtained from the definition found in the current literature, which differs from the compatible orientation by a multiple of $(-1)^{1+0}=-1$.} \label{fig: example 2d dual}
\end{figure}
\end{example}

Using Definition \ref{def: boundary of dual}, we can now resolutely extend the exterior derivative to $$\oplus_kC^{n-k}(\dual K_h)$$ by defining it as the homomorphism satisfying
\begin{equation*}
\langle \dif\hodge_h\omega_h,\dual\tau\rangle =\langle \hodge_h\omega_h,\partial\dual\tau\rangle,\,\,\,\,\,\forall\,\tau\in K_h.
\end{equation*}

\section{Consistency}
\label{s:consistency}

In this section,
we develop a framework to address consistency questions for operators such as the Hodge-Laplace operator.
We begin by proving a preliminary lemma which reduces the question of consistency of the Hodge-Laplace operator to the one of consistency of the Hodge star.
We then proceed with a demonstration of the latter, and derive bounds on the norms of various discrete operators.

We denote by $C^r\Lambda^k(\mathcal{O})$ the space of differential $k$-forms $\omega$ on an open set $\mathcal{O}\subset\mathbb{R}^n$ such that the function
\begin{equation*}
x\mapsto \omega_x(X_1(x),...,X_k(x))
\end{equation*}
 is $r$ times continuously differentiable for every $k$-tuples $X_1,...,X_k$ of smooth vector fields on $\mathcal{O}$. Given an $n$-polytope $\mathcal{Q}$, say an $n$-dimensional subcomplex of $K_h$, we write $C^r\Lambda^k(\mathcal{Q})$ for the set of differential $k$-forms obtained by restricting to $\mathcal{Q}$ a differential form in $C^r\Lambda^k(\mathcal{O})$, where $\mathcal{O}$ is any open set in $\mathbb{R}^n$ containing $\mathcal{Q}$. The `exact' discrete representatives for solutions of the continuous problems corresponding to (\ref{eq: Discrete Poisson Dirichlet}) are constructed using the {\em deRham map} $R_h$ defined on $C\Lambda^k(\mathcal{P})$ by
 \begin{equation*}
 \langle R_h\omega, s \rangle = \int_{s}\omega,
 \end{equation*}
 where $s$ is a $k$-dimensional simplex in $K_h$ or the dual of a $(n-k)$-dimensional one.
The next lemma depends on the extremely convenient property that $\dif_{\,h}R_h= R_h\dif\,$.
\begin{lemma}\label{lem: reducing consistency}
Given $\omega\in C^2\Lambda^k(\mathcal{P})$, we have
\begin{multline*}
\Delta_h R_h\omega - R_h \Delta \omega = \star_h \dif_{\,h}\left(\star_h R_h-R_h \star \right)\dif\omega +\left(\star_h R_h - R_h\star\right)d\star d\omega \\
+\dif_{\,h}\left(\star_hR_h-R_h\star\right)d\star \omega +\dif_{\,h} \star_h \dif_{\,h} \left(\star_h R_h -R_h \star\right)\omega.
\end{multline*}
\end{lemma}
\begin{proof}
Consider the explicit expression
\begin{multline*}
\Delta_h R_h\omega - R_h \Delta \omega = \left(\delta_h \dif_{\,h}+\dif_{\,h} \delta_h\right)R_h\omega -R_h\left(\delta \dif+\dif \delta\right)\omega \\
= \left(\star_h \dif_{\,h} \star_h \dif_{\,h} R_h - R_h \star \dif \star \dif\right)\omega+ \left(\dif_{\,h} \star_h \dif_{\,h}\star_h R_h - R_h \dif\star \dif\star\right)\omega
\end{multline*}
of the discrete Hodge-Laplace operator. Substituting
$$
\star_h \dif_{\,h} R_h \omega = \star_h R_h \dif \omega= \star_h R_h \dif \omega -R_h \star \dif \omega+R_h \star \dif \omega= \left(\star_h R_h-R_h \star \right)\dif\omega +R_h \star \dif\omega
$$
in the first summand, we obtain
\begin{align*}
\star_h \dif_{\,h} \star_h \dif_{\,h} R_h \omega - R_h \star \dif \star \dif \omega
&= \star_h \dif_{\,h}\left(\star_h R_h-R_h \star \right)\dif\omega + \star_h \dif_{\,h} R_h \star d\omega - R_h \star \dif \star \dif\omega \\
&= \star_h \dif_{\,h}\left(\star_h R_h-R_h \star \right)\dif\omega +\left(\star_h R_h - R_h\star\right)\dif\star \dif\omega.
\end{align*}
The second summand can be rewritten as
\begin{align*}
\left(\dif_{\,h} \star_h \dif_{\,h}\star_h R_h - R_h \dif\star \dif\star\right)\omega &= \dif_{\,h}\left(\star_h \dif_{\,h}\star_h R_h - R_h \star \dif\star\right)\omega \\
&= \dif_{\,h}\left(\star_h \dif_{\,h}\star_h R_h-\star_h R_h \dif \star +\star_h R_h \dif \star-R_h \star \dif\star\right) \omega \\
&= \dif_{\,h}\left(\star_hR_h-R_h\star\right)\dif\star \omega +\dif_{\,h}\left(\star_h \dif_{\,h} \star_h R_h - \star_h R_h \dif\star\right)\omega.
\end{align*}
The desired equality then follows from the identity
\begin{equation*}
\left(\star_h \dif_{\,h} \star_h R_h - \star_h R_h \dif\star\right)\omega = \left(\star_h \dif_{\,h} \star_h R_h - \star_h \dif_{\,h} R_h\star\right)\omega = \star_h \dif_{\,h} \left(\star_h R_h -R_h \star\right)\omega ,
\end{equation*}
which completes the proof.
\end{proof}

Given a $n$-simplex $\sigma$ and a $k$-dimensional face $\tau\prec\sigma$, the subspaces $P(\tau)$ and $P(\dual\tau)$ of $\mathbb{R}^n$ are perpendicular and satisfy $P(\tau)\times P(\dual\tau)=P(\sigma)$. Since the the discrete hodge star maps $k$-cochains on $K_h$ to $(n-k)$-cochains on the dual mesh $\dual K_h$, it captures the geometric gist of its classical Hodge operator $\star$, which acts on wedge products of orthonormal covectors as $\star (\dif\lambda_{\rho(1)}\wedge \dif\lambda_{\rho(2)}\wedge...\wedge \dif\lambda_{\rho(k)}) =\sgn(\rho)\,\dif\lambda_{\rho(k+1)}\wedge \dif\lambda_{\rho(k+2)}\wedge...\wedge \dif\lambda_{\rho(n)}$.

We hope to exploit this correspondence to estimate the consistency of $\star_h$. By definition, the deRham map directly relates the integral of a differential $k$-form $\omega$ over a $k$-dimensional simplex $\tau$ of $K_h$ with the value of its discrete representative at $\tau$. But one can also integrate any differential $(n-k)$-form, in particular $\star\omega$, over $\dual\tau$. This suggests that simple approximation arguments could be used to compare $\langle R_h\star\omega,\dual\tau\rangle$ with $\langle\star_h\omega_h,\dual\tau\rangle$, since the latter is defined as a scaling of the aforesaid integral $\omega_h=R_h\omega$ by the area factor $\vert\dual \tau\vert/\vert \tau\vert$, which is also naturally related to those very integrals under consideration.

The technique we use is best illustrated in low dimensions. Let $p=(p_1,p_2)$ be a vertex of a triangulation $K_h$ embedded in $\mathbb{R}^2$. For a function $f$ on $\mathbb{R}^2$, we have $\hodge f=f\dif x\wedge \dif y$. On the one hand, $\langle R_h f, p \rangle$ is simply the evaluation of $f$ at $p$. One the other hand, if $f$ is differentiable, then a Taylor expansion around $p$ yields
\begin{align*}
\langle R_h\star f,\dual p \rangle &= \iint_{\dual p} f \dif A\\
&= \iint_{\dual p} f(p) + ((x,y)-(p_1,p_2))^TDf(p)+O(h^2)\dif A\\
&= \vert\dual p\vert f(p) + O(h^3).
\end{align*}
Relying on our convention that $\vert p\vert=1$, we conclude that
\begin{equation}\label{eq: consistency prototype estimate}
\langle R_h\star f,\dual p \rangle - \langle \star_hR_h f, \dual p \rangle = \langle R_h\star f,\dual p \rangle - \vert\dual p\vert f(p) = O(h^3).
\end{equation}

This argument is used in the next theorem as a prototype which we generalize in order to derive estimates for differential forms of higher order in arbitrary dimensions.

\begin{theorem}\label{thm:consistency estimate}
	Let $\sigma$ be a $n$-simplex, and suppose $\tau\prec\sigma$ is $k$-dimensional. Then
	\begin{equation*}
	\langle \hodge_h R_h \omega , \dual\tau \rangle = \langle R_h\hodge\omega,\dual\tau\rangle+ O\left(h^{n+1}/(\gamma_{\tau})^k\right),\,\,\,\,\,\omega\in C^1\Lambda^k(\sigma).
	\end{equation*}
\end{theorem}
\begin{proof}
	Denote by $\Sigma(k,n)$ the set of strictly increasing maps $\{1,...,k\}\longrightarrow\{1,...,n\}$. Since $P(\tau)\perp P(\dual\tau)$ and $P(\tau)\times P(\dual\tau)=\mathbb{R}^n$, we can choose orthonormal vectors $t_i=x_i-c(\tau)$, $x_i\in \mathbb{R}^n$, such that $\{t_i\}_{i=1,...,k}$ and $\{t_i\}_{i=k+1,..,n}$ are bases for $P(\tau)$ and $P(\dual\tau)$ respectively. Let $\dif\lambda_1,...,\dif\lambda_n$ be a basis for $\text{Alt}^n\mathbb{R}^n$ dual to $\{t_i\}_{i=1,...,n}$. We write $\vol_{\tau}$ for $\dif\lambda_1\wedge...\wedge\dif\lambda_k$ and $\vol_{\dual\tau}$ for $\dif\lambda_{k+1}\wedge...\wedge\dif\lambda_n$.

	Following the argument which led to (\ref{eq: consistency prototype estimate}), we approximate the integral $\langle R_h\omega,\tau\rangle$ of a differential $k$-form $\omega=\sum_{\rho\in\Sigma(k,n)}f_{\rho}(\dif\lambda)_\rho$ over $\tau$ using Taylor expansion. This yields
	\begin{equation*}
	\int_{\tau}\sum_{\rho\in\Sigma(k,n)}f_{\rho}(\dif\lambda)_\rho
	=\int_{\tau} f_{1,...,k}\vol_{\tau}
	=\vert \tau \vert f_{1,...,k}\left(c(\tau)\right)+O(h^{k+1}).
	\end{equation*}
	We find similarly that
	\begin{equation*}
	\langle R_h\hodge \omega, \dual\tau \rangle
	= \int_{\dual \tau}\sum_{\rho\in\sum(k,n)}f_{\rho}\hodge(\dif\lambda)_{\rho}
	=\int_{\dual\tau} f_{1,...,k}\vol_{\dual\tau}
	=\vert\dual\tau\vert f_{1,...,k}\left(c(\tau)\right)+O(h^{n-k+1}).
	\end{equation*}
	Combining these two equations, we get
	\begin{align*}
	\langle R_h\hodge \omega, \dual\tau \rangle &=\frac{\vert\dual\tau\vert}{\vert\tau\vert}\langle R_h\omega,\tau\rangle +O\left(h^{n+1}/(\gamma_{\tau})^k\right)+O(h^{n-k+1})\\
	&=\langle\hodge_h R_h\omega,\dual\tau\rangle+O\left(h^{n+1}/(\gamma_{\tau})^k\right) ,
	\end{align*}
which completes the proof.
\end{proof}

\begin{corollary}\label{cor: consistency of hodge star on dual}
	Integrating on the dual mesh, we obtain
	\begin{equation*}
		\langle\star_h R_h (\star\omega),\tau\rangle = \langle R_h\star(\star\omega), \tau\rangle + O(h^{k+1}),\,\,\,\,\, \omega\in C^1\Lambda^k(\mathcal{P}),
	\end{equation*}
	when $K_h$ is regular.
\end{corollary}
\begin{proof}
	Using Theorem \ref{thm:consistency estimate}, we find that
	\begin{align*}
		\langle\star_h R_h (\star\omega),\tau\rangle &= (-1)^{k(n-k)}\langle\star_h R_h (\star\omega),\dual\dual\tau\rangle\\
		&= (-1)^{k(n-k)}\left(\frac{\vert\dual\tau\vert}{\vert\tau\vert}\right)^{-1}\langle R_h (\star\omega),\dual\tau\rangle\\
		&= (-1)^{k(n-k)}\left(\frac{\vert\dual\tau\vert}{\vert\tau\vert}\right)^{-1}\left(\frac{\vert\dual\tau\vert}{\vert\tau\vert}\langle R_h\omega,\tau\rangle +O(h^{n-k+1})\right)\\
		&= \langle R_h \star \left(\star\omega\right),\tau\rangle + O(h^{k+1}).
	\end{align*}
\end{proof}

\begin{corollary}\label{cor: consistency of hodge star}
The estimate
\begin{equation*}
\| \hodge_h R_h \omega - R_h \hodge \omega \|_\infty = O(h^{n-k+1}),\,\,\,\,\, \omega\in C^1\Lambda^k(\mathcal{P})
\end{equation*}
holds when $K_h$ is regular.
\end{corollary}

\begin{corollary}\label{cor: consistency of hodge star L2}
	The estimate
	\begin{equation*}
	\| \hodge_h R_h \omega - R_h \hodge \omega \|_h= O(h),\,\,\,\,\, \omega\in C^1\Lambda^k(\mathcal{P})
	\end{equation*}
	holds when $K_h$ is regular.
\end{corollary}
\begin{proof}
Using Theorem \ref{thm:consistency estimate}, we estimate directly
\begin{align*}
\| \hodge_h R_h \omega - R_h \hodge \omega \|_h^2 &= \sum_{\tau\in\Delta_k(K_h)}\left(\frac{\dual\tau}{\tau}\right)^{-1}\langle \hodge_h R_h \omega - R_h \hodge \omega,\dual\tau\rangle^2\\
&\leq C\sum_{\tau\in\Delta_k(K_h)}\left(\frac{\dual\tau}{\tau}\right)^{-1}\left( h^{n-k+1}\right)^2\\
& \lesssim \sum_{\tau\in\Delta_k(K_h)} h^{n+2}\\
&\lesssim h^2,
\end{align*}
where the last inequality is obtained from the fact that $\#\{\tau:\tau\in\Delta_k(K_h)\}\sim h^n$, which is a consequence of the regularity assumption.
\end{proof}

Finally, the following Corollary is of special importance. It's proof is similar to the one used for Corollary \ref{cor: consistency of hodge star L2}.

\begin{corollary} \label{cor: consistency of hodge star L2 on the dual}
		The estimate
	\begin{equation*}
	\| \left(\hodge_h R_h - R_h \hodge\right) \left(\star\omega\right) \|_h= O(h),\,\,\,\,\, \omega\in C^1\Lambda^k(\mathcal{P})
	\end{equation*}
	holds when $K_h$ is regular.
\end{corollary}
\begin{proof}
	\begin{align*}
	\| \left(\hodge_h R_h \omega - R_h \hodge\right) \left(\star\omega\right) \|^2_h &=  \sum_{\tau\in\Delta_k(K_h)}\langle\left(\hodge_h R_h \omega - R_h \hodge\right) \left(\star\omega\right),\tau\rangle \langle \star_h\left(\hodge_h R_h \omega - R_h \hodge\right) \left(\star\omega\right),\dual\tau\rangle\\
	&=\sum_{\tau\in\Delta_k(K_h)}\frac{\vert\dual\tau\vert}{\vert\tau\vert}\langle\left(\hodge_h R_h \omega - R_h \hodge\right) \left(\star\omega\right),\tau\rangle^2\\
	&\lesssim \sum_{\tau\in\Delta_k(K_h)} h^{n-2k}\left(h^{k+1}\right)\\
	&\lesssim h^2.
	\end{align*}
\end{proof}

 As claimed, we see that estimates on the norm of the operators defined in Section \ref{subs: Discrete Operators} would finally allow one to use the above corollaries to find related bounds on the consistency expression given in Lemma \ref{lem: reducing consistency}. In particular, since
 \begin{equation}\label{eq: Hodge-Laplace on functions}
 \Delta_h\omega_h=\delta_h\dif_{\,h}\omega_h,\,\,\,\,\, \omega_h\in C^0(K_h),
 \end{equation}
 it is sufficient for the Hodge-Laplace problem on $0$-forms to evaluate the operator norm of the discrete exterior derivative and of the discrete Hodge star over $C^k(\dual K_h)$, $k=1,...,n$.
Namely, from (\ref{eq: Hodge-Laplace on functions}) and the proof of Lemma \ref{lem: reducing consistency}, the consistency error of the Hodge-Laplace restricted to $0$-forms $\omega$ can be written as
\begin{equation*}
\Delta_h R_h\omega - R_h \Delta \omega = \star_h \dif_{\,h}\left(\star_h R_h-R_h \star \right)\dif\omega +\left(\star_h R_h - R_h\star\right)\dif\star \dif\omega.
\end{equation*}
The following estimates mainly rely on Lemma \ref{lem: bounded number of simplices}, and on the regularity assumption imposed on $K_h$.  Indeed, for $\omega_h \in \Delta_k(K_h)$, the inequality
\begin{equation}\label{eq: op maximum norm exterior derivative over dual}
\|\dif_{\,h}\hodge_h\omega_h\|_{\infty}\leq C\|\hodge_h\omega_h\|_\infty,
\end{equation}
is a direct consequence of the former, whereas
\begin{equation}\label{eq: op maximum norm Hodge star over dual}
\|\hodge_h\hodge_h\omega_h\|_{\infty}\sim h^{2k-n}\|\hodge_h\omega_h\|_{\infty}
\end{equation}
is immediately obtained from the latter. While
\begin{equation}\label{eq: op L2 norm Hodge star over dual}
\|\hodge_h\|_{\mathcal{L}\left((C^k(\dual K_h),\|\cdot\|_h),(C^{n-k}(K_h),\|\cdot\|_h)\right)}= 1
\end{equation}
is merely a statement about the Hodge star being an isometry, the last required estimate is derived from
 \begin{align}
 \|\dif_{\,h} \hodge_h\omega_h\|_h^2 & = \sum_{\eta\in\Delta_{k-1}(K_h)}\left(\frac{\vert\dual\eta\vert}{\vert\eta\vert}\right)^{-1}\langle\hodge_h\omega_h,\partial\dual\eta\rangle^2  \nonumber\\
 &=\sum_{\eta\in\Delta_{k-1}(K_h)}\,\,\sum_{\substack{\tau\in\Delta_k(K_h),\\\tau\succ\eta}}\left(\frac{\vert\dual\eta\vert}{\vert\eta\vert}\right)^{-1}\langle\hodge_h\omega_h,\dual\tau\rangle^2  \nonumber\\
 &\leq Ch^{-2}\|\hodge_h\omega_h\|_h. \label{eq: op L2 norm exterior derivative over dual}
 \end{align}
Finally, by combining the estimates (\ref{eq: op maximum norm exterior derivative over dual}), (\ref{eq: op maximum norm Hodge star over dual}), (\ref{eq: op L2 norm Hodge star over dual}) and (\ref{eq: op L2 norm exterior derivative over dual}),
with corollaries \ref{cor: consistency of hodge star on dual}, \ref{cor: consistency of hodge star}, \ref{cor: consistency of hodge star L2} and \ref{cor: consistency of hodge star L2 on the dual},
we infer
\begin{align*}
\|\star_h \dif_{\,h}\left(\star_h R_h-R_h \star \right)\dif\omega\|_h &\leq \|\star_h\|_{\mathcal{L}\left((C^k(\dual K_h),\|\cdot\|_h),(C^{n-k}(K_h),\|\cdot\|_h)\right)}\|\dif_{\,h}\left(\star_h R_h-R_h \star \right)\dif\omega\|_h\\
&\lesssim h^{-1} \|\left(\star_h R_h-R_h \star \right)\dif\omega\|_h\\
&\lesssim h^{-1}h = O(1),
\end{align*}

\begin{align*}
\|\star_h \dif_{\,h}\left(\star_h R_h-R_h \star \right)\dif\omega\|_{\infty} &\lesssim h^{2k-n}\|\dif_{\,h}\left(\star_h R_h-R_h \star \right)\dif\omega\|_{\infty}\\
&\lesssim h^{2k-n} \|\left(\star_h R_h-R_h \star \right)\dif\omega\|_{\infty}\\
& \lesssim h^{2k-n}h^{n-k} = O(h^{k}),
\end{align*}

\begin{equation*}
\|\left(\star_h R_h - R_h\star\right)\dif\star \dif\omega\|_h = O(h),
\end{equation*}
and
\begin{align*}
\|\left(\star_h R_h - R_h\star\right)\dif\star \dif\omega\|_\infty = O(h^{k+1}).
\end{align*}


We have proved the following.

\begin{corollary}\label{c: consistency bound}
If $K_h$ is regular, then
\begin{equation*}
\begin{split}
\Delta_h R_h\omega - R_h \Delta \omega
&= \star_h \dif_{\,h}\left(\star_h R_h-R_h \star \right)\dif\omega +\left(\star_h R_h - R_h\star\right)\dif\star \dif\omega \\
&= O(1)+O(h) ,
\end{split}
\end{equation*}
for $\omega\in C^2\Lambda^0(\mathcal{P})$,
in both the maximum and discrete $L^2$-norm.
\end{corollary}

Note that in the displayed equation above, we have used the notation $O(1)+O(h)$ to convey information on the individual sizes of the two terms appearing in the right hand side of the first line.
In fact, the preceding corollary will not be employed in what follows, since it yields the consistency error of $O(1)$.
However, the techniques we have used to estimate the individual terms and operator norms will prove to be fruitful.

\section{Stability and convergence}
\label{s:stability}

We carry on with a proof of the first order convergence estimates
\begin{equation*}
\|u_h-R_h u\|_h = O(h)\,\,\,\,\,\text{and}\,\,\,\,\, \|\dif\,(u_h-R_h u)\|_h=O(h),
\end{equation*}
previously stated in (\ref{intro convergence}). The demonstration is based on the discrete variational form of the Poisson problem (\ref{eq: Discrete Poisson Dirichlet}). Its expression can be derived efficiently using the next proposition.

\begin{proposition}\label{prop: adjoint of exterior derivative}
	If $\omega_h\in C^{k}(K_h)$ and $\eta_h\in C^{k+1}(K_h)$, then
	\begin{equation}
	\left(\dif_{\,h}\omega_h,\eta_h \right)_h = \left( \omega_h,\delta_h\eta_h\right)_h.
	\end{equation}
\end{proposition}
\begin{proof}
	It is sufficient to consider the case where $\omega_h(\tau)=1$ on a $k$-simplex $\tau$ in $K_h$ and vanishes otherwise. On the one hand,
	\begin{equation}\label{adjoint derivative equation}
	\left(\dif_{\,h}\omega_h,\eta_h \right)_h = \sum_\sigma\langle\omega_h, \partial\sigma\rangle\langle \hodge_h\eta_h, \dual\sigma\rangle =\langle\omega_h,\tau\rangle \sum_{\sigma\succ\tau}\langle\hodge\eta_h,\dual\sigma\rangle,
	\end{equation}
	where $\sigma$ is a $(k+1$)-simplex oriented so that it is consistent with the induced orientation on $\tau$. On the other hand, since $\delta_h = (-1)^{k}\hodge_h^{-1}\dif_{\,h}\hodge_h$ on $k$-cochains, it follows from definition \ref{def: boundary of dual} that
	\begin{equation*}
	\left(\omega_h,\delta_h\eta_h\right)_h = (-1)^{k+1}\langle\omega_h,\tau\rangle\langle \dif_{\,h}\hodge_h\eta_h, \dual\tau\rangle = \langle\omega_h,\tau\rangle\sum_{\sigma\succ\tau}\langle \hodge\eta_h, \dual\sigma\rangle,
	\end{equation*}
	where $\sigma$ is oriented as in (\ref{adjoint derivative equation}).
\end{proof}
Before we do so however, we prove a discrete Poincar\'e-like inequality for $0$-cochains that is essential to the argument. The inequality is established by comparing the discrete norm of these cochains with the continuous $L^2$-norm of Whitney forms
\begin{equation*}
\phi_\tau = k!\sum_{i=0}^k (-1)^i\lambda_i\dif\lambda_1\wedge...\wedge\widehat{\dif\lambda_i}\wedge...\wedge\dif\lambda_k,\,\,\,\,\,\tau\in\Delta_k(K_h),
\end{equation*}
where $\lambda_i$ is the piecewise linear hat function evaluating to $1$ at the $i$th vertex of $\tau$ and $0$ at every other vertices of $K_h$. Setting
\begin{equation*}
W_h\omega_h = \sum_{\tau}\omega_h(\tau)\phi_\tau,\,\,\,\,\, \forall \tau\in\Delta_k(K_h),
\end{equation*}
for all $\omega_h\in C^k(K_h)$ defines linear maps $W_h$, called Whitney maps, which have the key property that $W_h\dif_{\,h}\omega_h=\dif W_h\omega_h$.
\begin{theorem}\label{thm: equivalence with Whitney forms}
	Let $K_h$ be a family of regular triangulations. There exist two positive constants $c_1$ and $c_2$, independent of $h$, satisfying
	\begin{equation*}
	c_1 \|\omega_h\|_h \leq \|W_h\omega_h\|_{L^2\Lambda^k(K_h)}\leq c_2\|\omega_h\|_h, \,\,\,\,\, \omega_h\in C^k(K_h).
	\end{equation*}
\end{theorem}
\begin{proof}
	We proceed using a scaling argument. Suppose $\sigma=[v_0,...,v_n]\in\Delta_n(K_h)$, and assume $\tau\prec\sigma$ is an orientation of the face $v_0,...,v_k$. Let $\{\lambda_i\}$ be the barycentric coordinate functions associated to the vertices of $\sigma$ (these are the piecewise linear hat functions used in the definition of Whitney forms). The $1$-forms $\dif\lambda_0,...,\widehat{\dif\lambda_{\ell}},...,\dif\lambda_n$ are a basis for $\text{Alt}^n(\mathbb{R}^n)$ dual to $t^{\ell}_i=v_i-v_{\ell}$, $i=1,...,n$, i.e. the vectors $t^{\ell}_1,...t^{\ell}_n$ form a basis for $P(\sigma)$ and satisfy $\dif\lambda_1\wedge...\wedge \widehat{\lambda_{\ell}}\wedge...\wedge\dif\lambda_n(t^{\ell}_1,...,t^{\ell}_n)=1$. In this setting, it follows in general that for any oriented face $\rho=[v_{\rho(0)},...,v_{\rho(m)}]$ of $\sigma$,
	\begin{multline*}
	\vert \rho \vert = \vert\det\left(v_{\rho(1)}-v_{\rho(0)},...,v_{\rho(m)}-v_{\rho(0)}\right)\vert\\
	=\vert\det\left(t^{\rho(0)}_1,...,t^{\rho(0)}_m\right)\vert
	=\pm_{\rho}\frac{1}{m!}\vol_\rho(t^{\rho(0)}_1,...,t^{\rho(0)}_n).
	\end{multline*}
In other words,
\begin{equation}\label{eq: change of basis simplex}
\dif\lambda_{\rho(1)}\wedge...\wedge \widehat{\dif\lambda_{\rho(0)}} \wedge...\wedge\dif\lambda_{\rho(m)}=\pm_{\rho}\frac{1}{\vert\rho\vert m!}\vol_{\rho},
\end{equation}
where $\pm_{\rho}$ depends on the orientation of $\rho$.

Now let $\hat{\sigma}$ denote the standard $n$-simplex in $\mathbb{R}^n$, and consider an affine transformation $F:\hat{\sigma}\longrightarrow\sigma$ of the form $F=B+b$, where $B:\mathbb{R}^n\longrightarrow\mathbb{R}^n$ is a linear map and $b\in\mathbb{R}^n$ a vector. We compute the pullback $\hat{\phi_{\tau}}=F^*\phi_{\tau}$ of a Whitney $k$-form $\phi_{\tau}$ using (\ref{eq: change of basis simplex}). To lighten notation, we write $\hat{\tau}=F^{-1}(\tau)$ and $B_{\hat{\tau}}=B\vert_{P(\hat{\tau})}$. We evaluate directly
	\begin{align*}
	\left(F^* \phi_\tau\right)_x &= k!\sum_{i=0}^k \frac{\pm_i\lambda_{i}\circ F}{k!\vert\tau\vert}F^*\vol_{\tau} \\
	&= k!\sum_{i=0}^k (-1)^i\left(\lambda_{i}\circ F\right) \vert\det B_{\hat{\tau}}\vert\bigwedge_{\substack{\ell=0\\ \ell\neq i}}^k \dif\lambda_{\ell} \\
	&= \vert\det B_{\hat{\tau}}\vert\left(\phi_{\tau}\right)_{F(x)}.
	\end{align*}
A change of variables then yields
	\begin{align*}
	\int_{\sigma}\|\sum_{\tau}\omega_h(\tau)\phi_{\tau}\|^2 \vol_{\sigma}&= \int_{\sigma}\|\sum_{\tau}\frac{\omega_h(\tau)}{\vert \det B_{\hat{\tau}}\vert}(\hat{\phi}_{\tau})_{F^{-1}}\|^2\vol_{\sigma}\\
	&= \vert\det B\vert\int_{\hat{\sigma}}\|\sum_{\tau}\frac{\omega_h(\tau)}{\vert \det B_{\hat{\tau}}\vert}\hat{\phi}_{\tau}\|^2\vol_{\hat{\sigma}}.
	\end{align*}
	Using the equivalence of norms on finite dimensional Banach spaces and the regularity assumption on $K_h$, we therefore conclude that
	\begin{equation*}
	\|W_h\omega_h\|^2_{L^2\Lambda^k(\sigma)}\sim \vert\det B\vert\sum_{\tau}\left(\frac{\omega_h(\tau)}{\vert \det B_{\hat{\tau}}\vert}\right)^2
	\sim h^{n-2k}\sum_{\tau}\omega_h(\tau)^2
	\sim \|\omega_h\|_h^2,
	\end{equation*}
	where the equivalences do not depend on $h$.
\end{proof}

\begin{corollary}\label{cor: poincare}
	There exists a constant $C$, independent of $h$, such that the discrete Poincar\'e inequality
	\begin{equation*}
	\|\omega_h\|_h \leq C\|\dif_{\,h}\omega_h\|_h
	\end{equation*}
	holds for all $\omega_h\in C^0(K_h)$ such that $\omega_h = 0$ on $\partial K_h$.
\end{corollary}
\begin{proof}
	Using Theorem \ref{thm: equivalence with Whitney forms} and the Poincar\'e inequality, we have
	\begin{equation}
	\|\omega_h\|_h\lesssim\|W_h\omega_h\|_h\lesssim\|\dif W_h\omega_h\|_{L^2\Lambda^k(K_h)}=\| W_h\dif_{\,h}\omega_h\|_{L^2\Lambda^k(K_h)}\lesssim\|\dif_{\,h}\omega_h\|_h ,
	\end{equation}
	which establishes the proof.
\end{proof}

The argument is twofold. We first restrict our attention in (\ref{eq: Discrete Poisson Dirichlet}) to the homogeneous boundary condition $g=0$, and introduce in a second breath inhomogeneous Dirichlet conditions. Consider the following variational formulations.

Suppose that $\nu_h\in C^0(K_h)$ vanishes everywhere but at a vertex $p$. For all $\omega_h \in C^0(K_h)$, we have
\begin{equation*}
\left(\delta_h\dif_{\,h}\omega_h,\nu_h\right)_h - \left( R_h f,\nu_h\right)_h = \nu_h(p)\vert\dual p\vert\left( \langle\delta_h\dif_{\,h}\omega_h,p\rangle - \langle R_hf,p\rangle\right).
\end{equation*}
In other words, $\Delta_h\omega_h = R_h f$ if and only if $\left( \delta_hd_h\omega_h,\nu_h\right)_h = \left( R_h f,\nu_h\right)_h$ for all $\nu_h$. By Proposition \ref{prop: adjoint of exterior derivative}, we may thus equivalently find a discrete function $\omega_h$ vanishing on $\partial K_h$ such that
\begin{equation}\label{eq: critical solutions}
\left(\dif\omega_h,\dif_{\,h}\nu_h\right) = \left(R_h f,\nu_h\right)_h
\end{equation}
for all $0$-cochains $\nu_h$ satisfying the homogeneous boundary condition. Note that this problem in turn reduces to the one of minimizing the energy functional
\begin{equation*}\label{eq: energy}
E_h(\nu_h)= \frac{1}{2}\left(\dif_{\,h}\nu_h,\dif_{\,h}\nu_h\right)_h-\left(R_hf,\nu_h\right)_h
\end{equation*}
over the same collection of $\nu_h$. Indeed a direct computation shows that $\omega_h$ satisfies (\ref{eq: critical solutions}) if and only if $\frac{\dif}{\dif\epsilon}\big\vert_{\epsilon = 0}E_h\left(\omega_h+\epsilon \nu_h\right) = 0$ for all these $\nu_h$, in which case we further conclude from
\begin{equation*}
E_h\left(\omega_h+\nu_h\right)= \underbrace{\frac{1}{2}\left(\dif_{\,h}\omega_h,\dif_{\,h}\omega_h\right)_h}_{\geq 0}+\underbrace{\left(\omega_h,\nu_h\right)_h-\left(R_hf,\omega_h\right)_h}_{= 0\text{ by (\ref{eq: critical solutions})}}+\underbrace{\frac{1}{2}\left(\nu_h,\nu_h\right)_h-\left(R_hf,\nu_h\right)_h}_{E_h(\nu_h)}
\end{equation*}
that $\omega_h$ is a minimizer. Since $\left(\dif_{\,h}u_h,\dif_{\,h}u_h\right)_h=0$ if and only if $u_h$ is constant, $\delta_h\dif_{\,h}$ is invertible over the space of discrete functions vanishing on $\partial K_h$. The existence and uniqueness of a solution to (\ref{eq: critical solutions}) is thus guaranteed for all $h$.

Using Corollary \ref{cor: poincare} in (\ref{eq: critical solutions}) and assuming $R_hf$ is not identically $0$ (the case with trivial solution $\omega_h=0$), we obtain
\begin{equation*}
\left(\dif_{\,h}\omega_h,\dif_{\,h}\omega_h\right)_h\leq\|\omega_h\|_h\|R_hf\|_h\leq C\|\dif_{\,h}\omega_h\|_h\|R_hf\|_h,
\end{equation*}
and another application of that corollary finally yields the stability estimate
\begin{equation*}
\|\omega_h\|_h\leq C \|R_hf\|_h.
\end{equation*}

Unfortunately, convergence cannot be obtained with an application of the Lax-Richtmyer theorem with this estimate and the consistency bound derived in Corollary \ref{c: consistency bound}. However, Proposition \ref{prop: adjoint of exterior derivative} comes to our rescue.  Given a solution $\omega$ to the associated continuous problem, it allows us to use the expression for $\Delta_h$ given in Lemma \ref{lem: reducing consistency} to evaluate the error $e_h=R_h\omega-\omega_h$ with
\begin{align*}
\left(\dif_{\,h}e_h,\dif_{\,h}e_h\right)_h&=\left(\Delta_he_h,e_h\right)_h\\
&= \left(\star_h \dif_{\,h}\left(\star_h R_h-R_h \star \right)\dif\omega,e_h\right)_h +\left(\left(\star_h R_h - R_h\star\right)\dif\star \dif\omega,e_h\right)_h\\
&=\left(\star_h^{-1}\left(\star_h R_h-R_h \star \right)\dif\omega,\dif_{\,h}e_h\right)_h + \left(\left(\star_h R_h - R_h\star\right)\dif\star \dif\omega,e_h\right)_h\\
&\leq Ch\|\dif_{\,h} e\|_h+Ch\|e_h\|_h,
\end{align*}
where the last inequality was obtained from an application of Corollary \ref{cor: consistency of hodge star L2}, (\ref{eq: op L2 norm Hodge star over dual}) and the second to last estimate in Corollary \ref{c: consistency bound}. Using the discrete Poincar\'e-like inequality again completes the proof of the homogeneous case of the following theorem.
\begin{mdframed}
\begin{theorem}
The unique discrete solution $\omega_h\in C^0(K_h)$ of problem (\ref{eq: Discrete Poisson Dirichlet}) stated for $0$-forms over a regular triangulation $K_h$ satisfy
\begin{equation*}
\|e_h\|_h\leq C\|\dif_{\, h}e_h\|_h= O(h).
\end{equation*}
\end{theorem}
\end{mdframed}

As one would expect, the case of inhomogeneous boundary conditions reduces to the homogeneous one. Given $g_h\neq \vec{0}$, the Poisson problem is to find $\omega_h$ in the affine space $\{g+\eta_h\vert\, \eta_h\in C^0(K_h), \eta_h = 0\, \text{on}\,\,\, \partial K_h\}$ satisfying (\ref{eq: critical solutions}). That is, we are looking for   $u_h\in C^0(K_h)\cap\{\eta_h:\eta_h\vert_{\partial K_h}=0\}$ satisfying
\begin{equation}\label{eq: critical solutions inhomogeneous}
\left(\dif_{\,h}u_h,\dif_{\,h}\nu_h\right)_h= \left(R_hf, \nu_h \right)_h - \left(\dif_{\, h}g_h,\dif_{\, h}\nu_h\right)_h = F(u_h,\nu_h),
\end{equation}
for all $\nu_h$ vanishing on the boundary, which is equivalent to the previous homogeneous problem, but with the linear functional $F:C^0(K_h)\times C^0(K_h)\longrightarrow \mathbb{R}$ on the right hand side. Repeating the previous arguments, we obtain the stability estimate
\begin{equation}
\|u_h\|_h\leq C\left( \|R_hf\|_h+\|\dif_{\, h} g_h\|_h\right).
\end{equation}
The rest of the proof goes through similarly, so that the solution $\omega_h=g_h+u_h$ to (\ref{eq: Discrete Poisson Dirichlet}) with inhomogeneous boundary condition $g$ is first order convergent.

\section{Numerical experiments}
\label{s:numerics}

In this section, we report on some numerical experiments performed over two and three dimensional triangulations. The discrete solutions are computed from the inverse of $\Delta_h$ built as a compound of the operators defined in Section \ref{subs: Discrete Operators}. The volumes needed to implement the Hodge stars as diagonal matrices $\hodge^h_k$ are computed in a standard way. The matrices $\dif^{\,h}_{\,k-1,\text{Primal}}=(\partial^h_{k})^T$ acting on vector representations of cochains (denoted by $[\cdot]$) in $C^{k-1}(K_h)$ are created using the algorithm suggested in \parencite{BH12}. By construction, the orientation of $\dif^{\,h}_{\,k-1,\text{Primal}}[\tau]\in\Delta_{k}(K_h)$ is therefore consistent with the orientation of the $(k-1)$-simplex $\tau$.  Hence, it is suitably oriented for definition \ref{def: boundary of dual} to imply
\begin{align*}
\left(\dif^{\,h}_{\,\,n-k,\text{Dual}}\hodge^h_k\lbrack\omega_h\rbrack\right)^T \lbrack\tau\rbrack&=
\langle \dif_{\,h}\hodge_h\omega_h, \dual\tau\rangle\\
&=(-1)^k\sum_{\eta\succ\tau}\langle\hodge_h\omega_h,\dual\eta\rangle\\
&= (-1)^k\left(\hodge_h^k[\omega_h]\right)^T\dif^{\,h}_{\,k-1,\text{Primal}}[\tau]\\
&=(-1)^k\left((\dif^{\,h}_{\,k-1,\text{Primal}})^T\hodge_h^k[\omega_h]\right)^T[\tau]
\end{align*}
for all $\omega_h\in C^k(K_h)$ and $\tau\in\Delta_{k-1}(K_h)$. This proves the practical definition
\begin{equation*}
\dif^{\,h}_{\,\,n-k,\text{Dual}}=(-1)^k\left(\dif^{\,h}_{\,k-1,\text{Primal}}\right)^T = (-1)^k \partial^h_k
\end{equation*}
given in \parencite{Desbrun2008}, which we use in the following.

\subsection{Experiments in two dimensions}
Convergence of DEC solutions for $0$-forms is first studied over two types of convex polygons. Non-convex polygons are then used to evaluate the consequences of a lack of regularity caused by re-entrant corners.

\subsubsection{Regular polygons}\label{Regular n-gons}
We conduct numerical experiments on refinements of the form $\{K_{c\cdot2^{-i}}\}_{i=0}^N$ of the wheel graph $W_{5}$ whose boundary is a regular pentagon. Triangulation of the pentagonal domain is performed by recursively subdividing each elementary $2$-simplex of the initial complex into its four inscribed subtriangles delimited by the graph of its medial triangle and its boundary. We illustrate the process in Figure \ref{fig: W5 Initial Mesh}.

A discrete solution to the Hodge-Laplace Dirichlet problem with trigonometric exact solution $u(x,y)=x^2\sin(y)$ over this complex is displayed along with its error function in Figure \ref{results convex polygons}. The size of the errors for its collection of refinements with $N=9$ is found in different discrete norms in Table \ref{table: Integral Errors Pentagon Trig}.
The results show second order convergence both in the discrete $L^2$ and $H^1$ norms.
Other cases where the initial primal triangulation is designed from a wheel graph on $n+1$ vertices are also considered. In particular, repeating the experiment on regular $n$-gons with $6\leq n \leq 8$ yields the same asymptotics.

\begin{figure}
	\centering
	\begin{subfigure}[b]{0.45\textwidth}
		\centering
		\includegraphics[width=\textwidth]{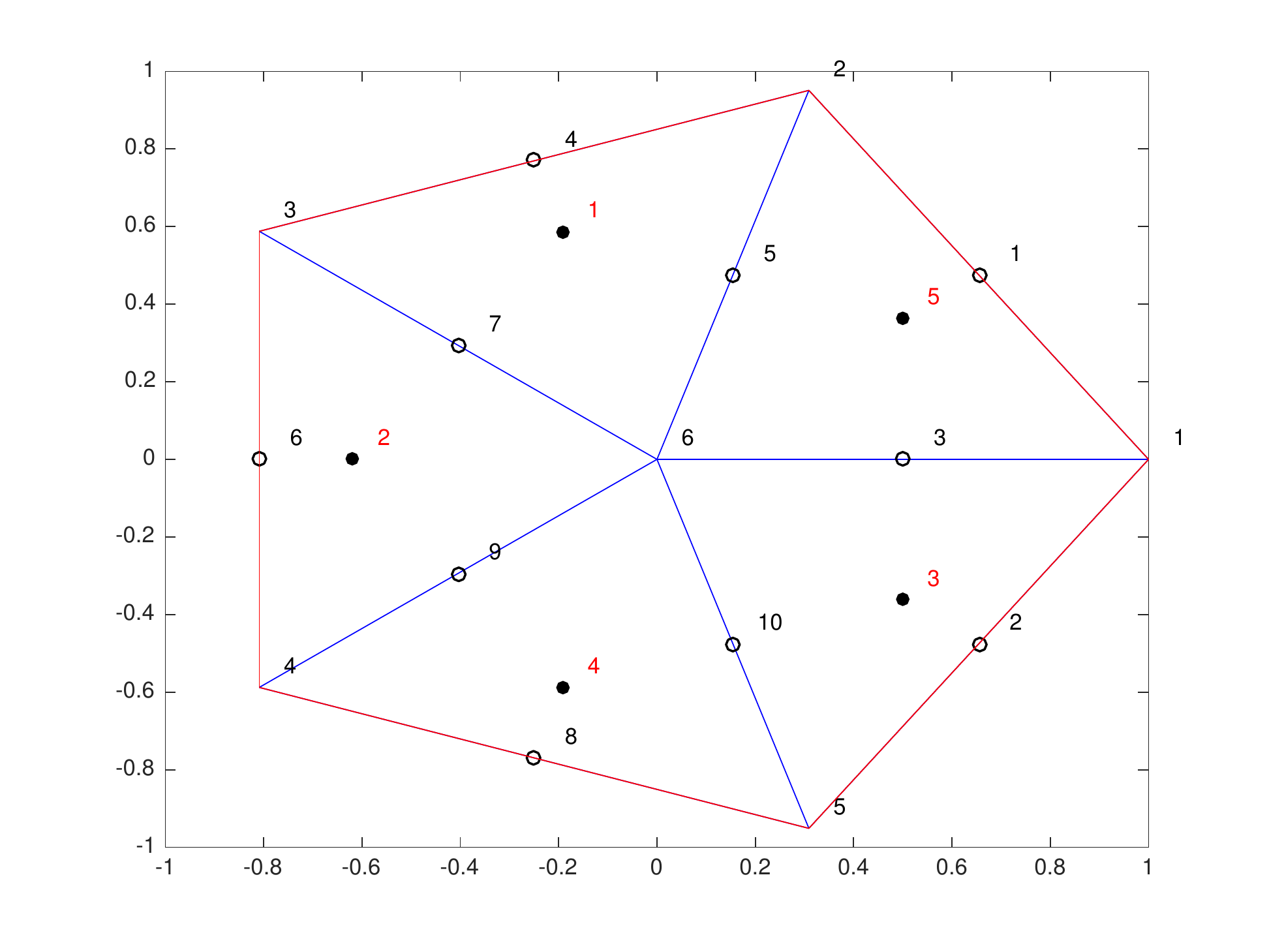}
		\caption{Initial complex $K_c$}
	\end{subfigure}
	\begin{subfigure}[b]{0.45\textwidth}
		\centering
		\includegraphics[width=\textwidth]{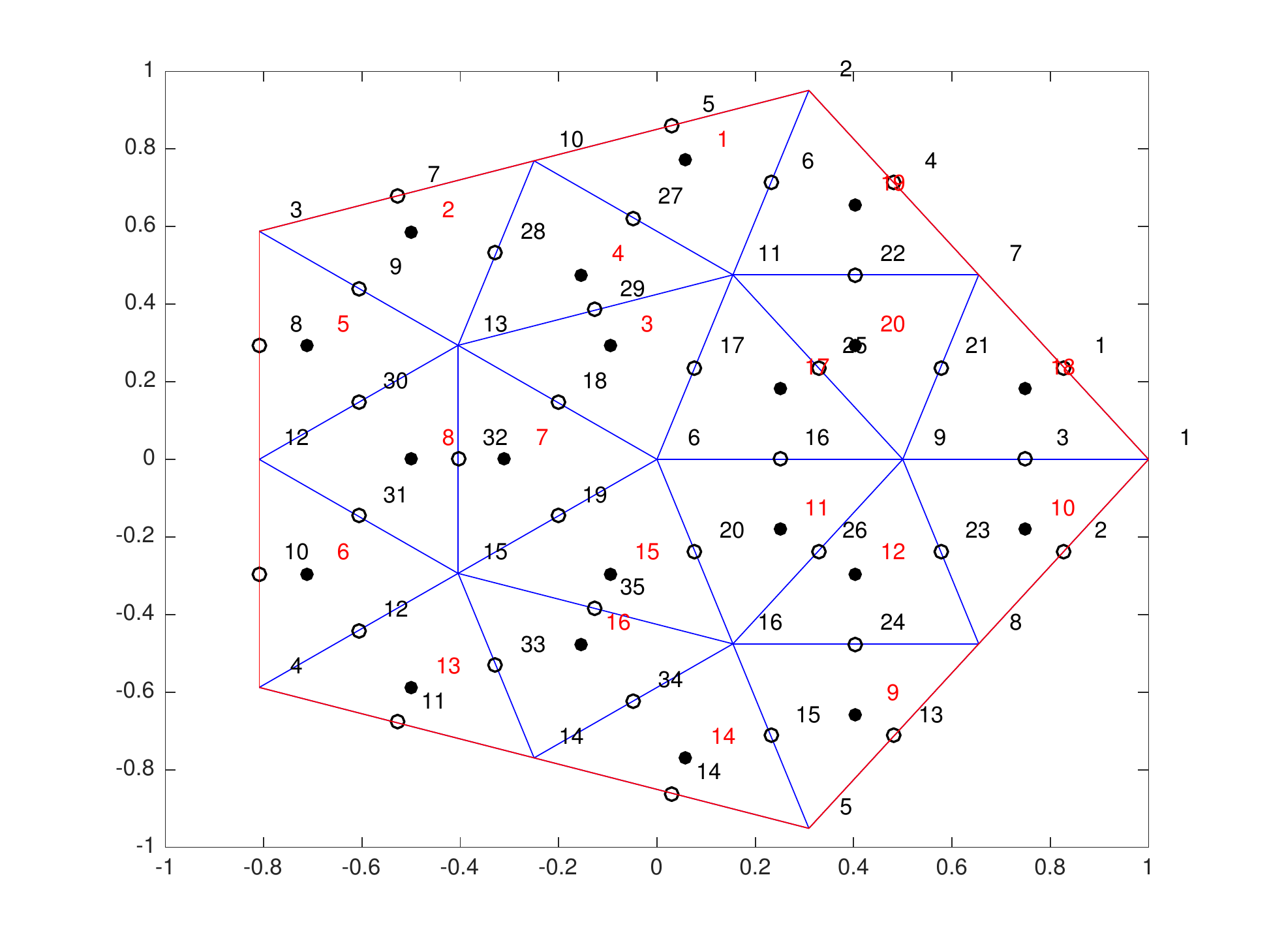}
		\caption{Refined complex $K_{c\cdot2^{-1}}$}
	\end{subfigure}
	\caption{The initial primal triangulation of $W_5$ and its first refinement are shown. On the left, (a) was obtained by fixing the length of the interior edges to $1$, consequently setting the initial maximal length for the edges to $c=(2-2\cos(2\pi/5))^{\frac{1}{2}}$. The triangles in $K_c$ are subdivided by the edges of their medial triangles to produce (b).}\label{fig: W5 Initial Mesh}
\end{figure}

\begin{figure}
	\centering
	\begin{subfigure}[b]{0.45\textwidth}
		\centering
		\includegraphics[width=\textwidth]{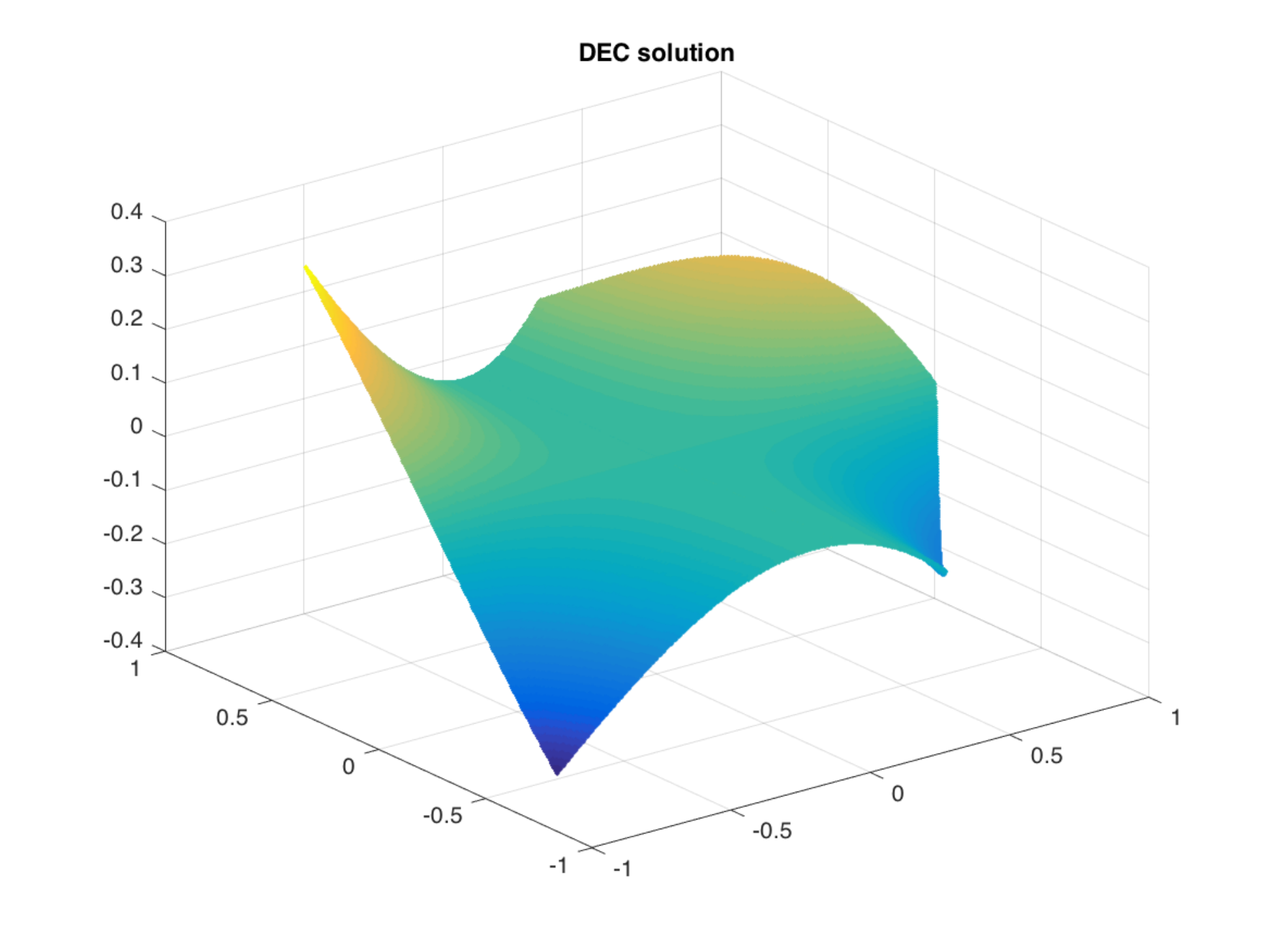}
		\caption{$u_{C\cdot 2^{-8}}$}
	\end{subfigure}
	\begin{subfigure}[b]{0.45\textwidth}
		\centering
		\includegraphics[width=\textwidth]{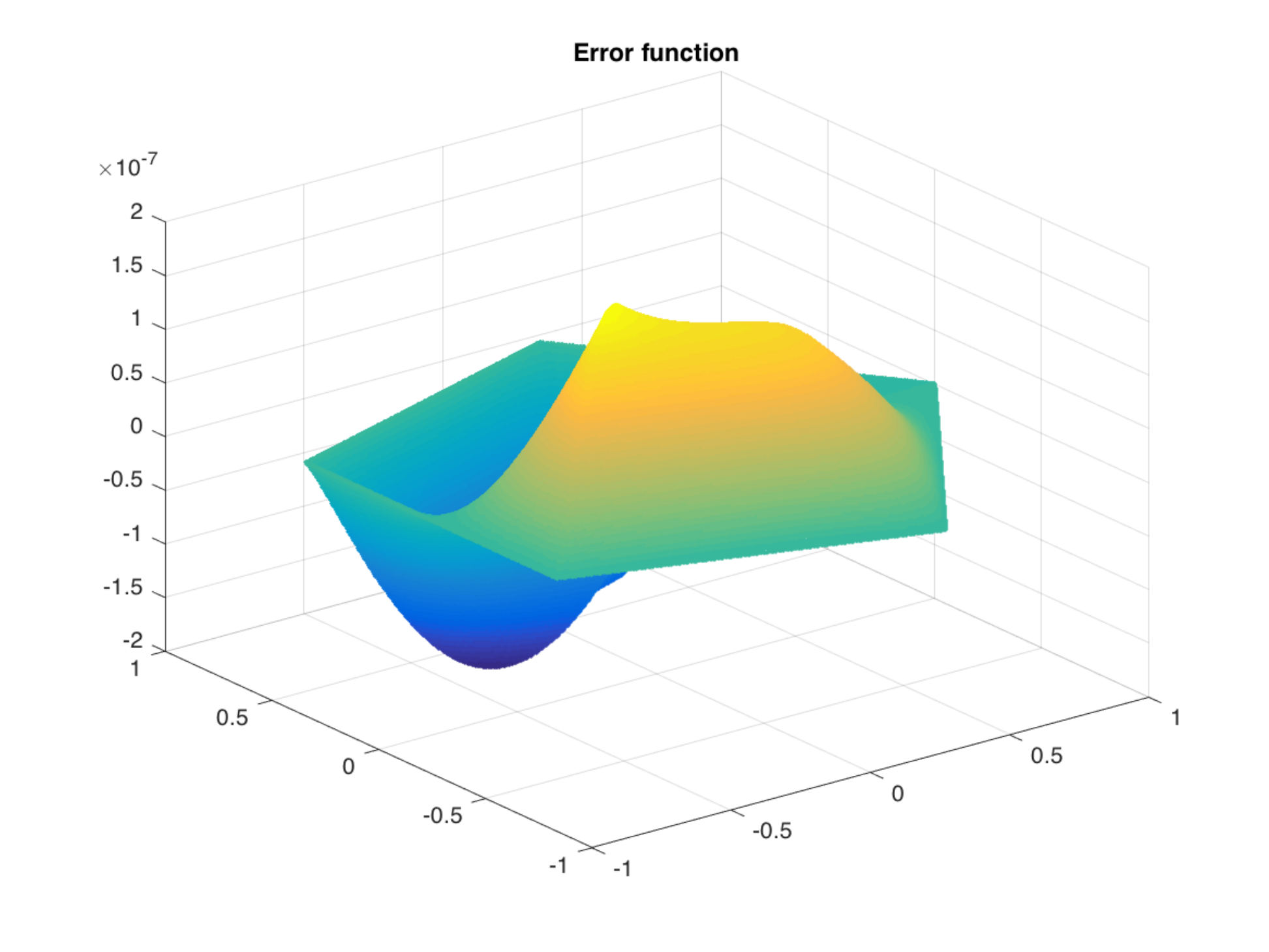}
		\caption{$e_{C\cdot 2^{-8}}$}
	\end{subfigure}
	\caption{This figure refers to the experiment $u(x,y)=x^2\sin(y)$. On the left, a DEC solution to the Hodge-Laplace Dirichlet problem over $T_{C\cdot 2^{-8}}$ is displayed. A plot of the related error function is found on the right.}\label{results convex polygons}
\end{figure}

\begin{table}
\scalebox{0.70}{
\begin{tabular}{|c||c|c|c|c|c|c|}
\hline
$i$ &$e_i^\infty=\|e_{C\cdot 2^{-i}}\|_{\infty}$& $\log\left(e_i^\infty/e^\infty_{i-1}\right)$ & $e^{H^1_d}=\|\dif e_{C\cdot 2^{-i}}\|_{C\cdot 2^{-i}}$& $\log\left(e_i^{H^1_d}/e^{H^1_d}_{i-1}\right)$ & $e^{L_d^2}=\|e_{C\cdot 2^{-i}}\|_{C\cdot 2^{-i}}$ & $\log\left(e_i^{L^2_d}/e^{L^2_d}_{i-1}\right)$ \\
\hline
0 & 1.561251e-17 & - & 2.975694e-17 & - & 1.487847e-17 & -\\
\hline
1 & 3.202794e-03 & -4.754362e+01 & 1.072846e-02 & -4.835714e+01 & 2.821094e-03 & -4.743002e+01\\
\hline
2 & 7.836073e-04 & 2.031128e+00 & 2.879579e-03 & 1.897512e+00 & 6.332754e-04 & 2.155350e+00\\
\hline
3 & 1.956510e-04 & 2.001848e+00 & 7.353114e-04 & 1.969431e+00 & 1.532456e-04 & 2.046987e+00\\
\hline
4 & 4.891893e-05 & 1.999818e+00 & 1.849975e-04 & 1.990850e+00 & 3.798925e-05 & 2.012183e+00\\
\hline
5 & 1.227086e-05 & 1.995157e+00 & 4.633277e-05 & 1.997401e+00 & 9.477213e-06 & 2.003057e+00\\
\hline
6 & 3.067823e-06 & 1.999949e+00 & 1.158895e-05 & 1.999283e+00 & 2.368052e-06 & 2.000762e+00\\
\hline
7 & 7.669629e-07 & 1.999987e+00 & 2.897627e-06 & 1.999806e+00 & 5.919350e-07 & 2.000190e+00\\
\hline
8 & 1.917491e-07 & 1.999937e+00 & 7.244331e-07 & 1.999948e+00 & 1.479789e-07 & 2.000047e+00\\
\hline
		\end{tabular}
	}
\caption{Experiment of Subsection \ref{Regular n-gons} with $u(x,y)=x^2\sin(y)$. The terms $C\cdot2^{-i}\in\mathbb{R}$ indicates the values of $h$ in $\|\cdot\|_h$.}\label{table: Integral Errors Pentagon Trig}
	\end{table}

\subsubsection{Rectangular domains}\label{Rectangular Domains}
Three different types of regular triangulations were used to investigate convergence of DEC solutions over the unit square. These appear as degenerate examples of circumcentric complexes in which the circumcenters lie on the hypotenuses of the primal triangles. Examples of these triangulations are shown in Figure \ref{fig: Triangulations of the Square}.

The computations carried over the unit square produce results similar to the ones previously obtained in Subsection \ref{Regular n-gons}.
\begin{figure}
	\begin{subfigure}{0.32\textwidth}
		\centering
		\includegraphics[width=\textwidth]{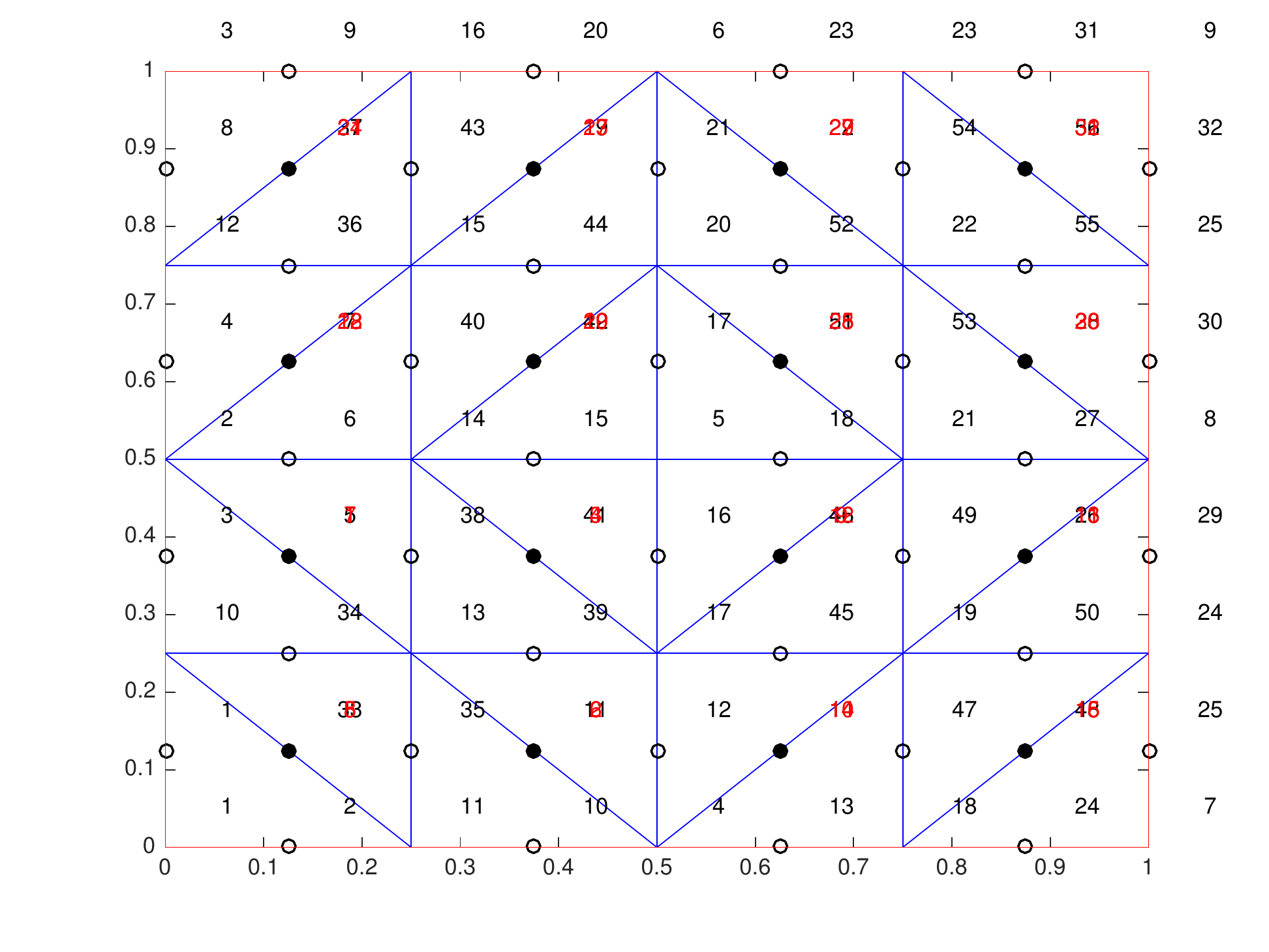}

	\end{subfigure}
	\begin{subfigure}{0.32\textwidth}
		\centering
		\includegraphics[width=\textwidth]{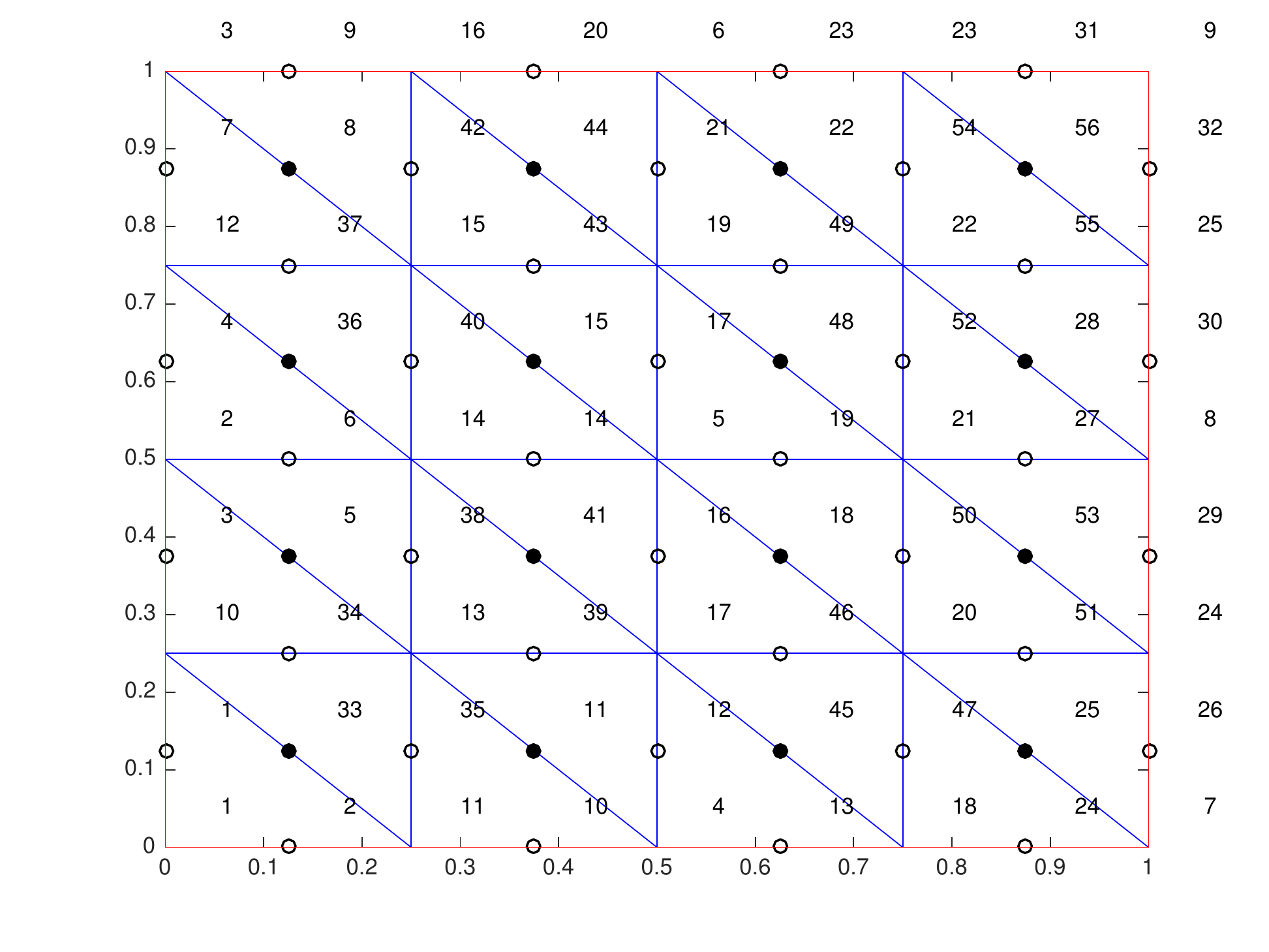}
	\end{subfigure}
	\begin{subfigure}{0.32\textwidth}
		\centering
		\includegraphics[width=\textwidth]{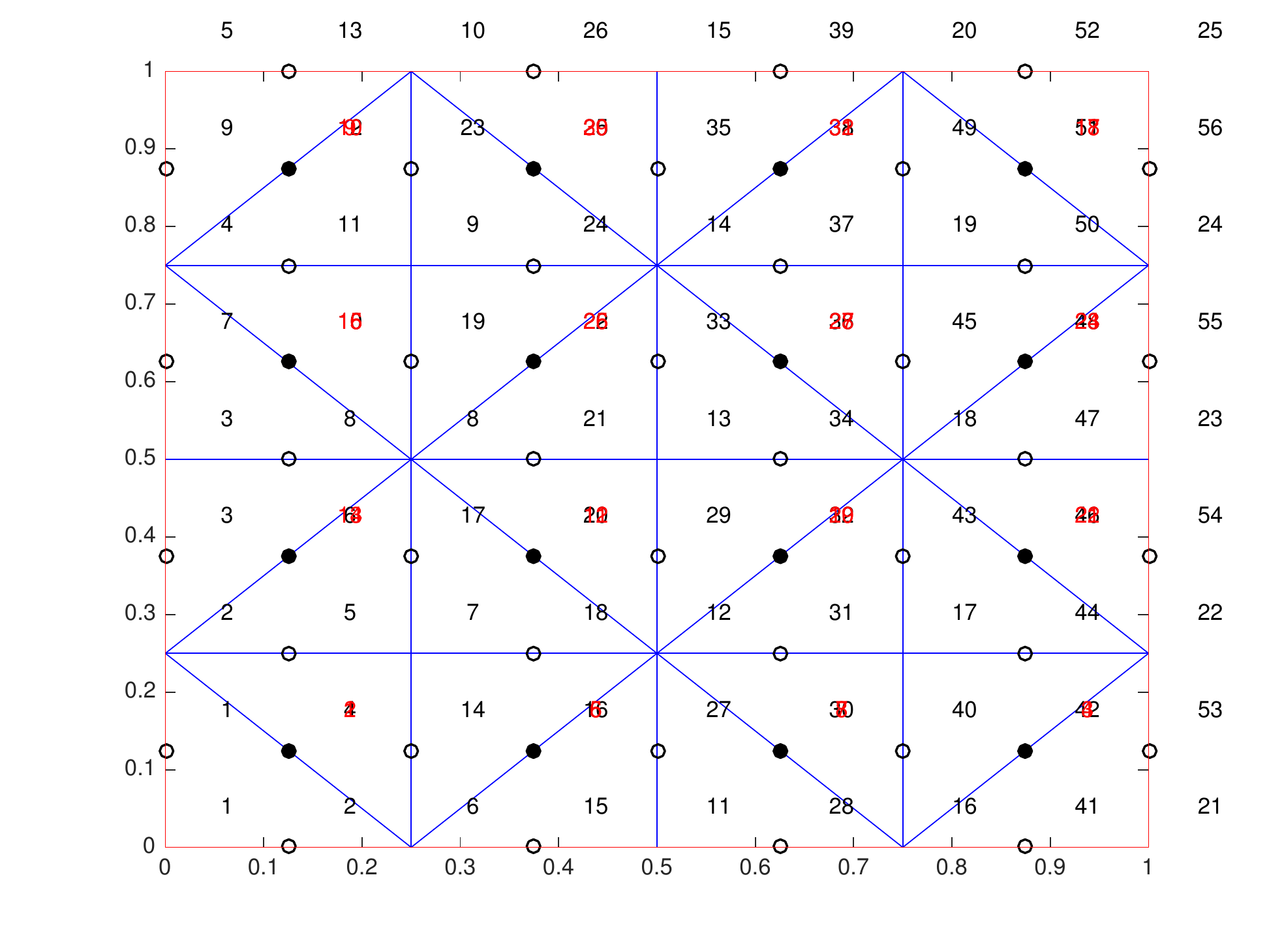}
	\end{subfigure}
	\caption{Three regular triangulations of the square.}\label{fig: Triangulations of the Square}
\end{figure}

\subsubsection{Unstructured meshes}\label{sub: Asymmetric Convex Polygons}
Second order convergence in all norm is also observed for unstructured meshes on convex polygons.
The discrete solutions behaved similarly over non-convex polygons when the solution is smooth enough.
One example of a primal domain over which the algorithm was tested is shown in Figure \ref{fig:asymmetric mesh}.

\begin{figure}
	\centering
	\begin{subfigure}[b]{0.45\textwidth}
		\centering
		\includegraphics[width=\textwidth]{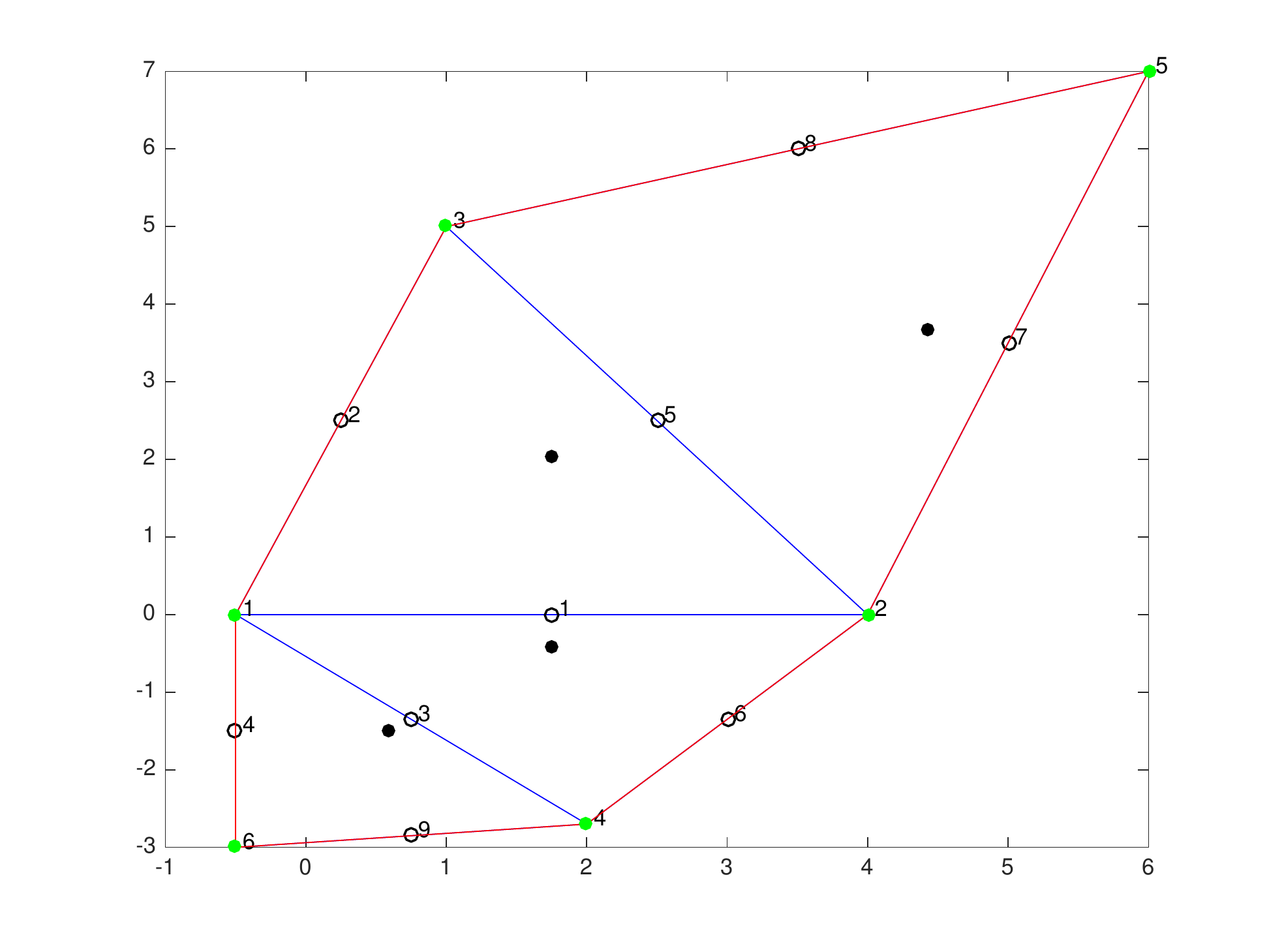}
		\caption{Initial complex $K_c$}
	\end{subfigure}
	\begin{subfigure}[b]{0.45\textwidth}
		\centering
		\includegraphics[width=\textwidth]{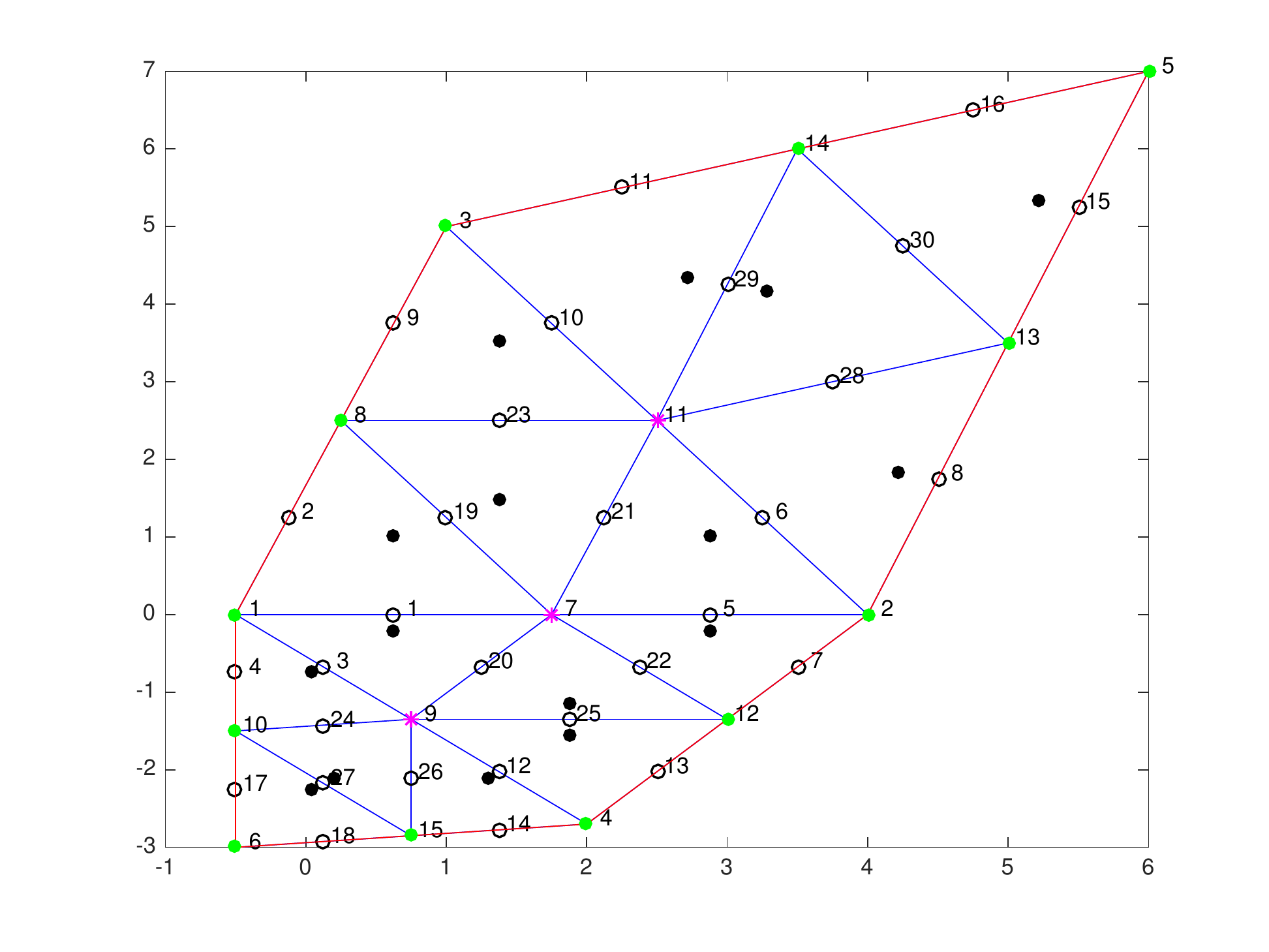}
		\caption{Refined complex $K_{c\cdot2^{-1}}$}
	\end{subfigure}
	\caption{Unstructured triangulations were refined as in Subsection \ref{Regular n-gons}.}\label{fig:asymmetric mesh}
\end{figure}

\subsubsection{Non-convex polygons}\label{concave Polygons}
We study the convergence behavior of DEC when a reentrant corner is present along the boundary of the primal complex.
More precisely, we consider re-entrant corners with various angles, leading to the exact solutions $u\in H^{1+\mu}(\mathcal{P})$, $0<\mu<1$,
which lack $H^2$ regularity.

Elected candidates for $u\in H^{1+\mu}(\mathcal{P})$ were harmonic functions of the form $r^{\mu}\sin(\mu\theta)$. These functions suit the model of Figure \ref{fig: Pentagon with Corner} (A) for $\mu=\pi/(2\pi -\beta)=\pi/\alpha$ and the result of the experiments can found in Table \ref{table: Integral Errors Pentagon Corner 5/8}.
The initial triangulation for this case is also plotted along its first refinement in Figure \ref{fig: Pentagon with Corner}. The refinement algorithm of Subsection \ref{Regular n-gons} was used. Other experiments were performed and led to the empirical observation that $\|u-u_h\|_h\sim h^{2\mu}$ and $\|\dif\,(u-u_h)\|_h\sim h^{\mu}$.
This convergence behavior is precisely what would be expected when a finite element method is applied to the same problem,
which may be explained by the fact that
our discretization is equivalent to a finite element method, cf. \parencite{HPW06,Wardetzky08}

\begin{figure}
	\begin{subfigure}{0.32\textwidth}
		\centering
			\definecolor{cqcqcq}{rgb}{0.7529411764705882,0.7529411764705882,0.7529411764705882}
			\definecolor{sqsqsq}{rgb}{0.12549019607843137,0.12549019607843137,0.12549019607843137}
			\definecolor{aqaqaq}{rgb}{0.6274509803921569,0.6274509803921569,0.6274509803921569}
			\scalebox{0.535}{
		\begin{tikzpicture}
		(14.605669383349145,2.902663975507868);
		\fill[color=aqaqaq,fill=aqaqaq,pattern=north east lines,pattern color=aqaqaq] (4.72,-3.1133515029206382) -- (8.61968766844833,-4.332418268388487) -- (10.984158821850333,-1.000307247602113) -- (8.545794691623108,2.2781173829998234) -- (4.674331628792105,0.9721842134802645) -- cycle;
		\draw [shift={(7.52,-1.02)},color=sqsqsq,fill=sqsqsq,fill opacity=0.5] (0,0) -- (72.9536800605884:0.6725886422328028) arc (72.9536800605884:359.78017786952046:0.6725886422328028) -- cycle;
		\fill[color=cqcqcq,fill=cqcqcq,fill opacity=1.0] (8.545794691623108,2.2781173829998234) -- (7.516251579186198,-1.0350368026550791) -- (11.,-1.0333515029206348) -- cycle;
		\draw [color=aqaqaq] (4.72,-3.1133515029206382)-- (8.61968766844833,-4.332418268388487);
		\draw [color=aqaqaq] (8.61968766844833,-4.332418268388487)-- (10.984158821850333,-1.000307247602113);
		\draw [color=aqaqaq] (10.984158821850333,-1.000307247602113)-- (8.545794691623108,2.2781173829998234);
		\draw [color=aqaqaq] (8.545794691623108,2.2781173829998234)-- (4.674331628792105,0.9721842134802645);
		\draw [color=aqaqaq] (4.674331628792105,0.9721842134802645)-- (4.72,-3.1133515029206382);
		\draw (6.216654682045458,-1.0961812246762428) node[anchor=north west] {$\alpha$};
		\draw [color=cqcqcq] (11.,-1.0333515029206348)-- (8.545794691623108,2.2781173829998234);
		\draw [line width=2.pt] (8.545794691623108,2.2781173829998234)-- (4.674331628792105,0.9721842134802644);
		\draw [line width=2.pt] (4.674331628792105,0.9721842134802644)-- (7.516251579186198,-1.0350368026550791);
		\draw [line width=2.pt] (4.674331628792105,0.9721842134802644)-- (4.719621231177869,-3.0794664855630187);
		\draw [line width=2.pt] (4.719621231177869,-3.0794664855630187)-- (8.61968766844833,-4.332418268388486);
		\draw [line width=2.pt] (7.516251579186198,-1.0350368026550791)-- (8.61968766844833,-4.332418268388486);
		\draw [line width=2.pt] (8.61968766844833,-4.332418268388486)-- (11.,-1.0333515029206348);
		\draw [shift={(7.516251579186198,-1.0350368026550791)},color=cqcqcq,fill=cqcqcq,fill opacity=1.0]  (0,0) --  plot[domain=4.837604210764684E-4:1.2695119391391207,variable=\t]({1.*3.483748828454465*cos(\t r)+0.*3.483748828454465*sin(\t r)},{0.*3.483748828454465*cos(\t r)+1.*3.483748828454465*sin(\t r)}) -- cycle ;
		\draw [line width=2.pt,dash pattern=on 3pt off 3pt] (7.516251579186198,-1.0350368026550791)-- (11.,-1.0333515029206348);
		\draw (10.129897691399947,-0.43226828150799318) node[anchor=north west] {$\Gamma$};
		\draw (7.855325192212651,-0.2009935674185522) node[anchor=north west] {$\beta$};
		\draw [line width=2.pt,dash pattern=on 3pt off 3pt] (8.545794691623108,2.2781173829998234)-- (7.516251579186198,-1.0350368026550791);
		\begin{scriptsize}
		\draw [fill=black] (8.545794691623108,2.2781173829998234) circle (2.5pt);
		\draw [fill=black] (7.516251579186198,-1.0350368026550791) circle (2.5pt);
		\draw [fill=black] (11.,-1.0333515029206348) circle (2.5pt);
		\draw [fill=black] (4.674331628792105,0.9721842134802644) circle (2.5pt);
		\draw [fill=black] (4.719621231177869,-3.0794664855630187) circle (2.5pt);
		\draw [fill=black] (8.61968766844833,-4.332418268388486) circle (2.5pt);
		\end{scriptsize}
		\end{tikzpicture}
	}
	\caption{Model of exact solutions}
	\end{subfigure}
	\begin{subfigure}{0.32\textwidth}
		\centering
		\includegraphics[width=\textwidth]{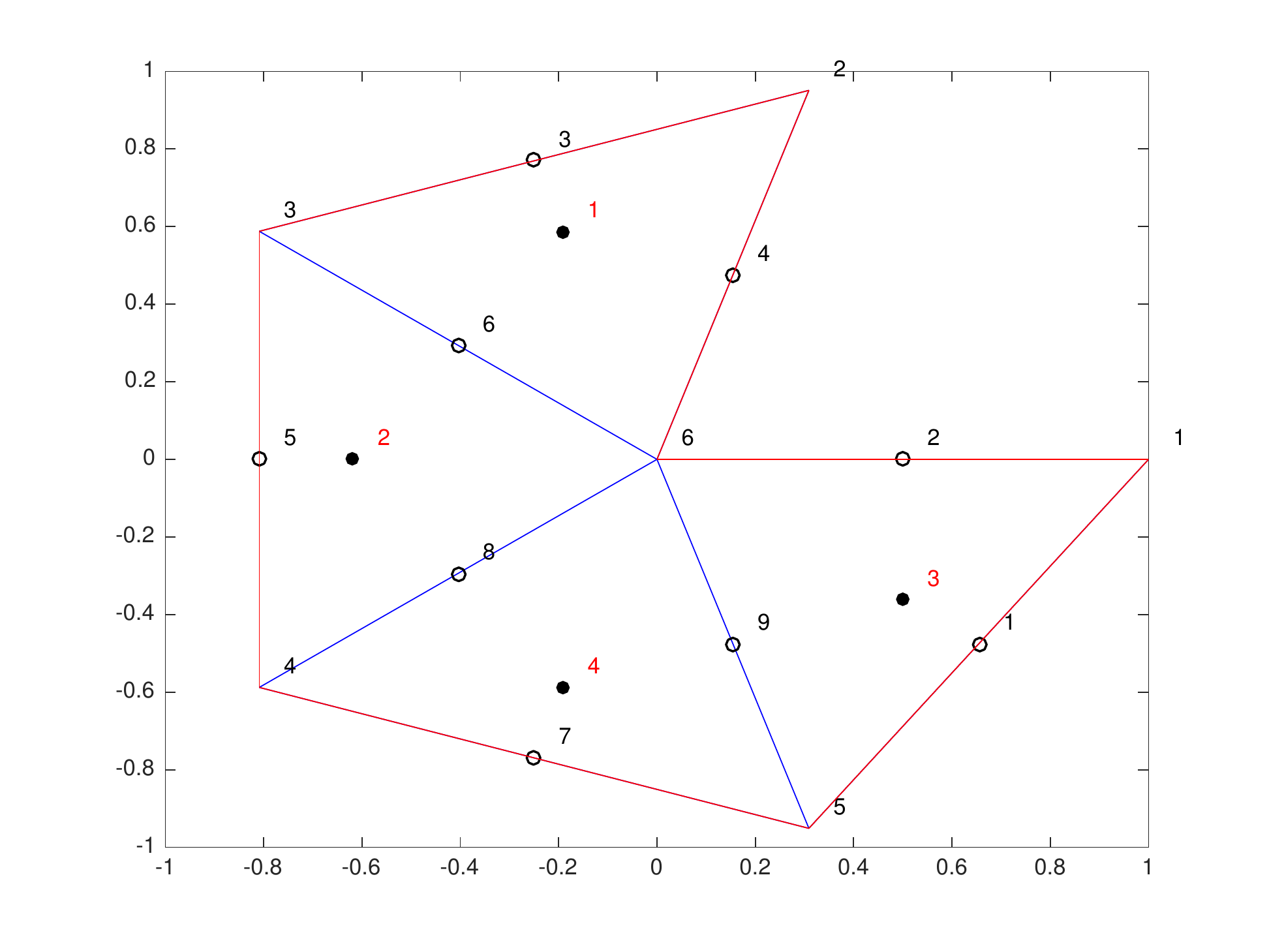}
		\caption{Initial complex $K_c$}
	\end{subfigure}
	\begin{subfigure}{0.32\textwidth}
		\centering
		\includegraphics[width=\textwidth]{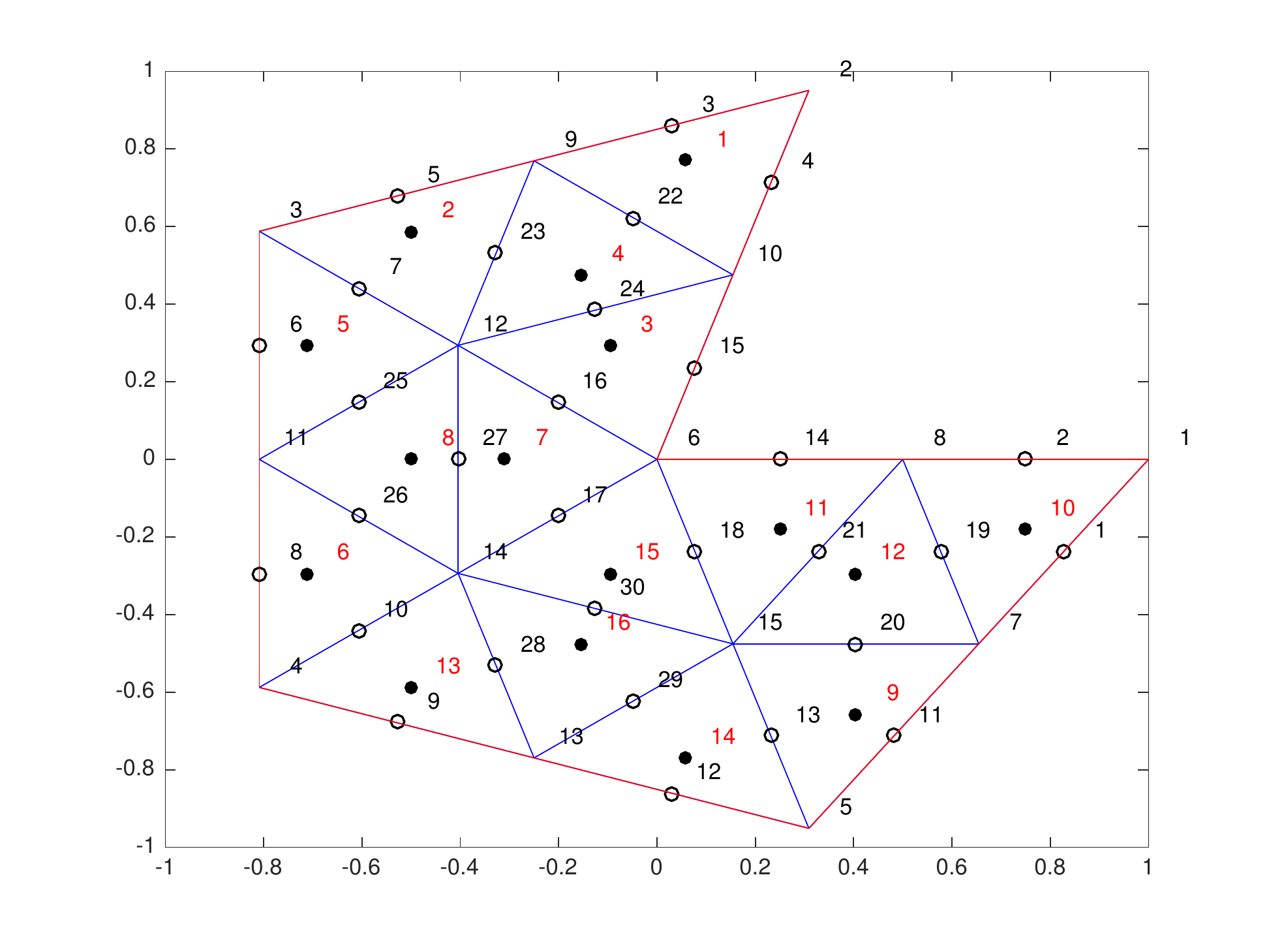}
		\caption{Refined complex $T_{c\cdot2^{-1}}$}
	\end{subfigure}
	\caption{We design an exact solution $u\in H^{1+\mu}(\mathcal{P})$, $0<\mu<1$, such that $\Delta u = 0$ in $\mathcal{P}$ and $u=0$ on the dotted piece of the boundary $\Gamma$ drawn in (A). The  red boundary shown in (B) and (C) corresponds to a concave cycle of $W_5$ on which calculations were conducted. This is an instance of the case where $\alpha=8\pi/5$ in (A). The constant $c$ was fixed as in Figure \ref{fig: W5 Initial Mesh}.}\label{fig: Pentagon with Corner}
\end{figure}
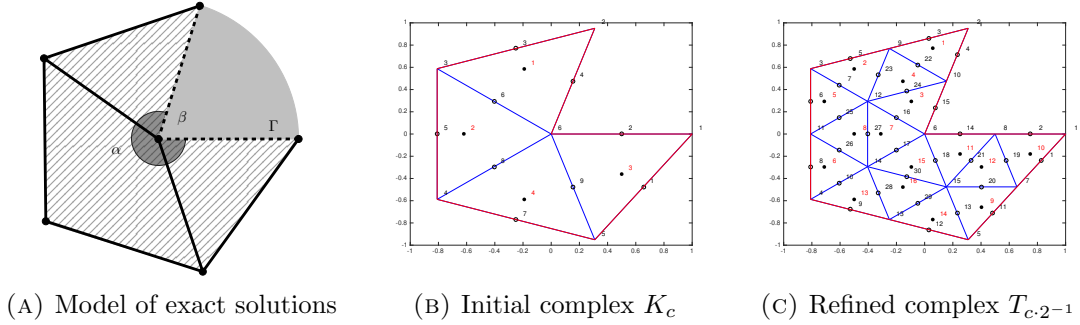

\begin{figure}
	\centering
	\begin{subfigure}{0.45\textwidth}
		\centering
		\includegraphics[width=\textwidth]{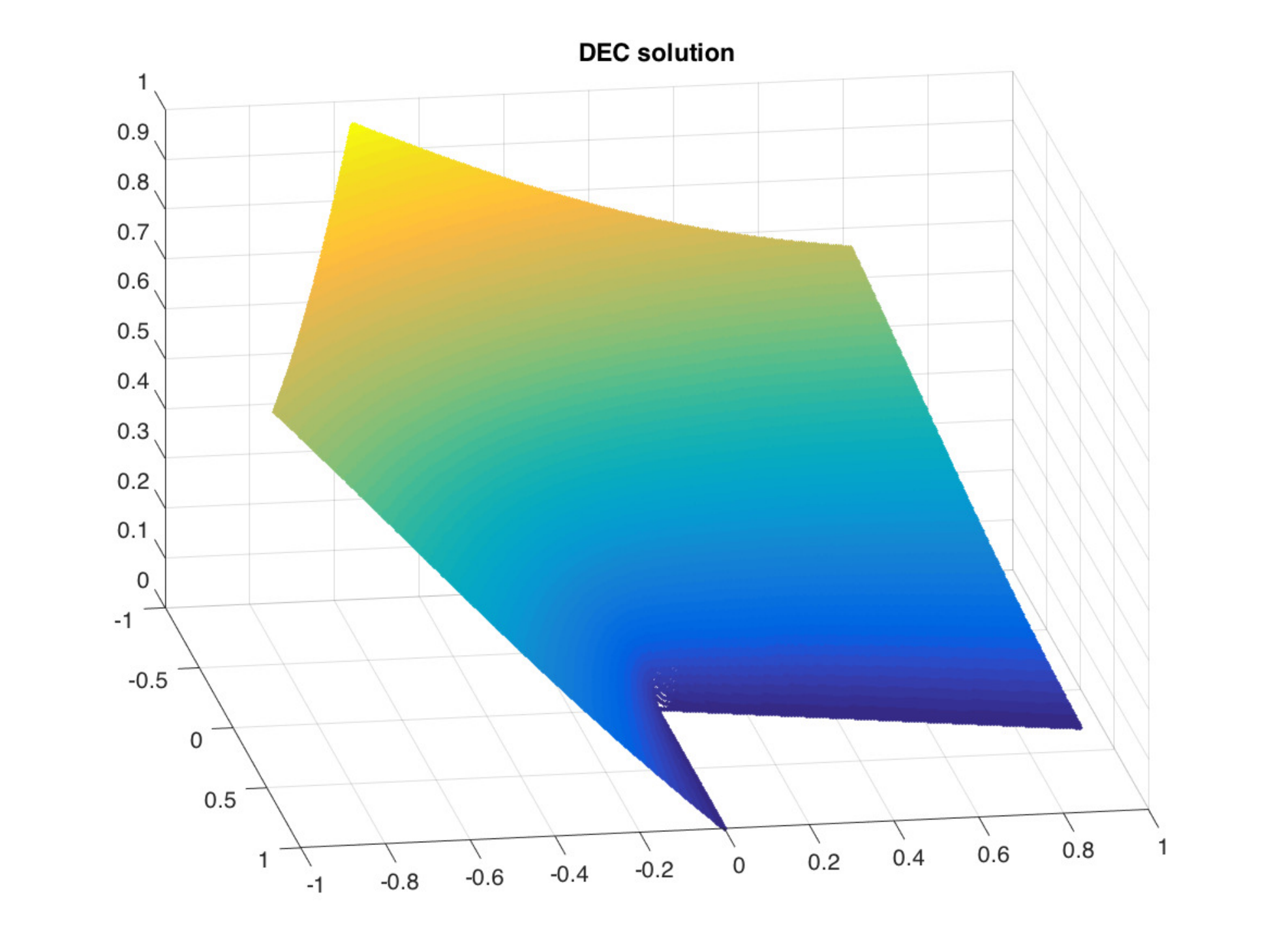}
		\caption{$u_{C\cdot 2^{-8}}$}
	\end{subfigure}
	\begin{subfigure}{0.45\textwidth}
		\centering
		\includegraphics[width=\textwidth]{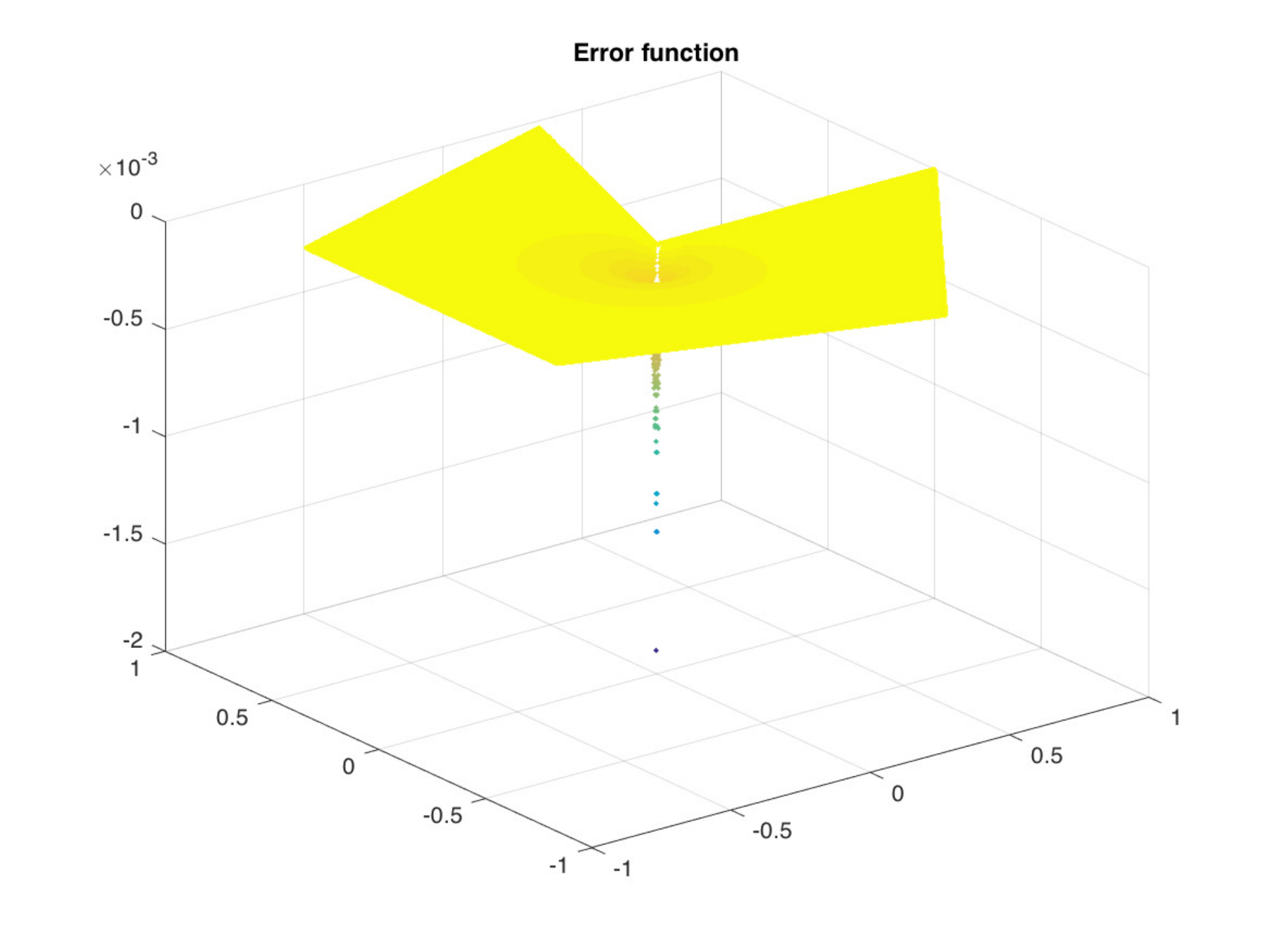}
		\caption{$e_{C\cdot 2^{-8}}$}
	\end{subfigure}
	\caption{These plots were obtained from the experiment $r^{\mu}\sin(\mu\theta)$, $\mu=5/8$. On the left, a DEC solution to the Hodge-Laplace Dirichlet problem over $K_{c\cdot 2^{-8}}$ is displayed. A plot of the related error function is found on the right.}\label{results nonconvex polygons}
\end{figure}

\begin{table}
	\scalebox{0.70}{
\begin{tabular}{|c||c|c|c|c|c|c|}
\hline
$i$ &$e_i^\infty=\|e_{C\cdot 2^{-i}}\|_{\infty}$& $\log\left(e_i^\infty/e^\infty_{i-1}\right)$ & $e^{H^1_d}=\|\dif e_{C\cdot 2^{-i}}\|_{C\cdot 2^{-i}}$& $\log\left(e_i^{H^1_d}/e^{H^1_d}_{i-1}\right)$ & $e^{L_d^2}=\| e_{C\cdot 2^{-i}}\|_{C\cdot 2^{-i}}$ & $\log\left(e_i^{L^2_d}/e^{L^2_d}_{i-1}\right)$ \\
\hline
0 & 0 & - & 0 & - & 0 & - \\
\hline
1 & 3.402738e-02 & - & 8.467970e-02 & - & 2.346479e-02 & - \\
\hline
2 & 3.194032e-02 & 9.131748e-02 & 6.533106e-02 & 3.742472e-01 & 1.353817e-02 & 7.934654e-01 \\
\hline
3 & 2.346298e-02 & 4.449927e-01 & 4.496497e-02 & 5.389676e-01 & 6.570546e-03 & 1.042947e+00 \\
\hline
4 & 1.595752e-02 & 5.561491e-01 & 2.983035e-02 & 5.920204e-01 & 2.970932e-03 & 1.145097e+00 \\
\hline
5 & 1.054876e-02 & 5.971636e-01 & 1.952228e-02 & 6.116590e-01 & 1.299255e-03 & 1.193231e+00 \\
\hline
6 & 6.894829e-03 & 6.134867e-01 & 1.270715e-02 & 6.194814e-01 & 5.584503e-04 & 1.218184e+00 \\
\hline
7 & 4.485666e-03 & 6.201927e-01 & 8.252738e-03 & 6.226958e-01 & 2.377754e-04 & 1.231830e+00 \\
\hline
8 & 2.912660e-03 & 6.229847e-01 & 5.354822e-03 & 6.240341e-01 & 1.007013e-04 & 1.239517e+00 \\
\hline
\end{tabular}
}
\caption{Experiment of Subsection \ref{concave Polygons} with $r^{\mu}\sin(\mu\theta)$, $\mu=5/8$. The terms $C\cdot2^{-i}\in\mathbb{R}$ indicates the values of $h$ in $\|\cdot\|_h$.}\label{table: Integral Errors Pentagon Corner 5/8}
\end{table}

\subsection{Three dimensional case}\label{sub: 3-Dimensional Triangulations}

As explained in \parencite{VanderZee2008}, it is not an easy task  in general to generate and refine tetrahedral meshes of $3$-dimensional arbitrary convex domains without altering its circumcentric quality. The next example is a degenerate case: the circumcenters of the tetrahedra used to triangulate the cube lies on lower dimensional simplices, but it does allow for a better control over the step size in $h$ (and thus for a better evaluation of the convergence rate), because the mesh can be refined recursively by gluing smaller copies of itself filling the cuboidal domain. The triangulation is displayed in Figure \ref{cubic mesh}. Again, a second order convergence in all norms similar to the results of Subsection \ref{Regular n-gons} is obtained and presented in Table \ref{table: errors cubic}.

\begin{table}
	\scalebox{0.70}{
		\begin{tabular}{|c||c|c|c|c|c|c|}
			\hline
			$i$ &$e_i^\infty=\|e_{C\cdot 2^{-i}}\|_{\infty}$& $\log\left(e_i^\infty/e^\infty_{i-1}\right)$ & $e^{H^1_d}=\|\dif e_{C\cdot 2^{-i}}\|_{C\cdot 2^{-i}}$& $\log\left(e_i^{H^1_d}/e^{H^1_d}_{i-1}\right)$ & $e^{L_d^2}=\|e_{C\cdot 2^{-i}}\|_{C\cdot 2^{-i}}$ & $\log\left(e_i^{L^2_d}/e^{L^2_d}_{i-1}\right)$ \\
\hline
0 & 8.586493e-04 & - & 1.487224e-03 & - & 3.035784e-04 & -\\
\hline
1 & 2.666725e-04 & 1.687000e+00 & 6.216886e-04 & 1.258358e+00 & 1.156983e-04 & 1.391702e+00 \\
\hline
2 & 7.122948e-05 & 1.904523e+00 & 1.774812e-04 & 1.808526e+00 & 3.166206e-05 & 1.869540e+00 \\
\hline
3 & 1.835021e-05 & 1.956678e+00 & 4.594339e-05 & 1.949737e+00 & 8.083333e-06 & 1.969733e+00 \\
\hline
4 & 4.621759e-06 & 1.989283e+00 & 1.158904e-05 & 1.987096e+00 & 2.031176e-06 & 1.992635e+00 \\
\hline
		\end{tabular}
	}
	\caption{Experiment of Subsection \ref{sub: 3-Dimensional Triangulations} with $u(x,y)=x^2\sin(y)+\cos(z)$.The terms $C\cdot2^{-i}\in\mathbb{R}$ indicates the values of $h$ in $\|\cdot\|_h$.}\label{table: errors cubic}
\end{table}

\begin{figure}
	\centering
	\begin{subfigure}[b]{0.45\textwidth}
		\centering
		\includegraphics[width=\textwidth]{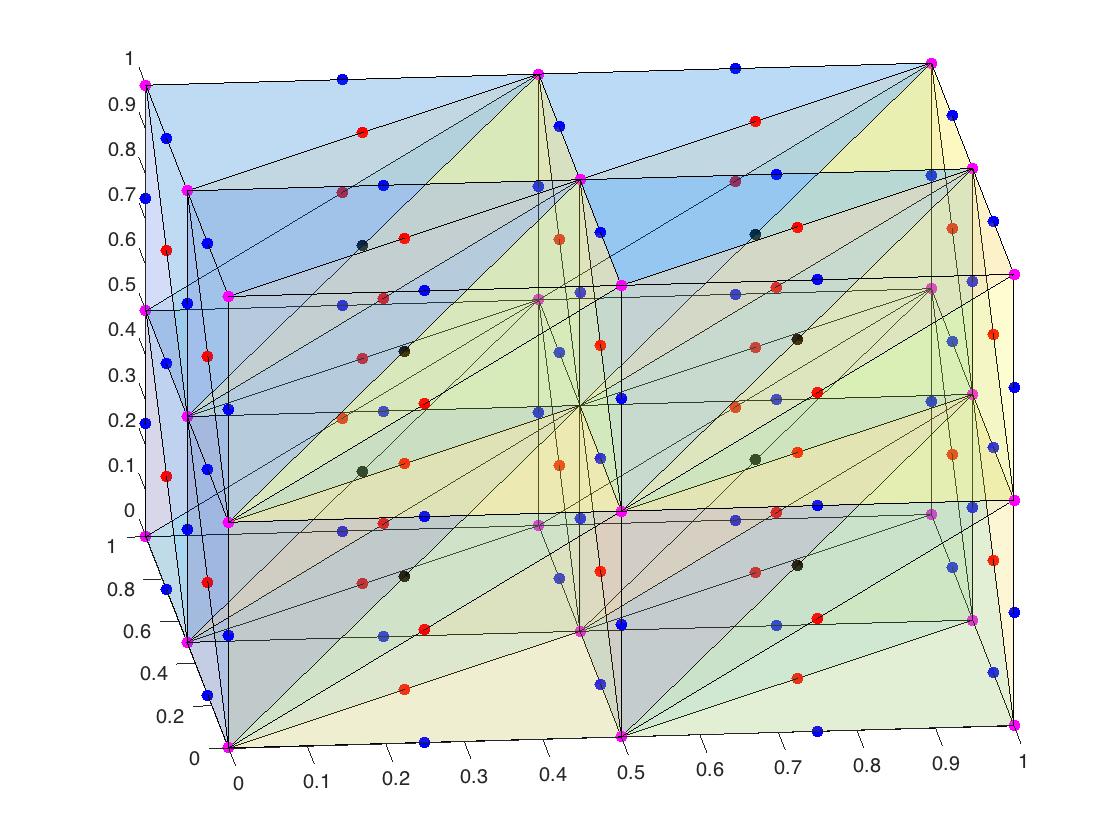}
		\caption{Initial complex $K_c$}
	\end{subfigure}
	\begin{subfigure}[b]{0.45\textwidth}
		\centering
		\includegraphics[width=\textwidth]{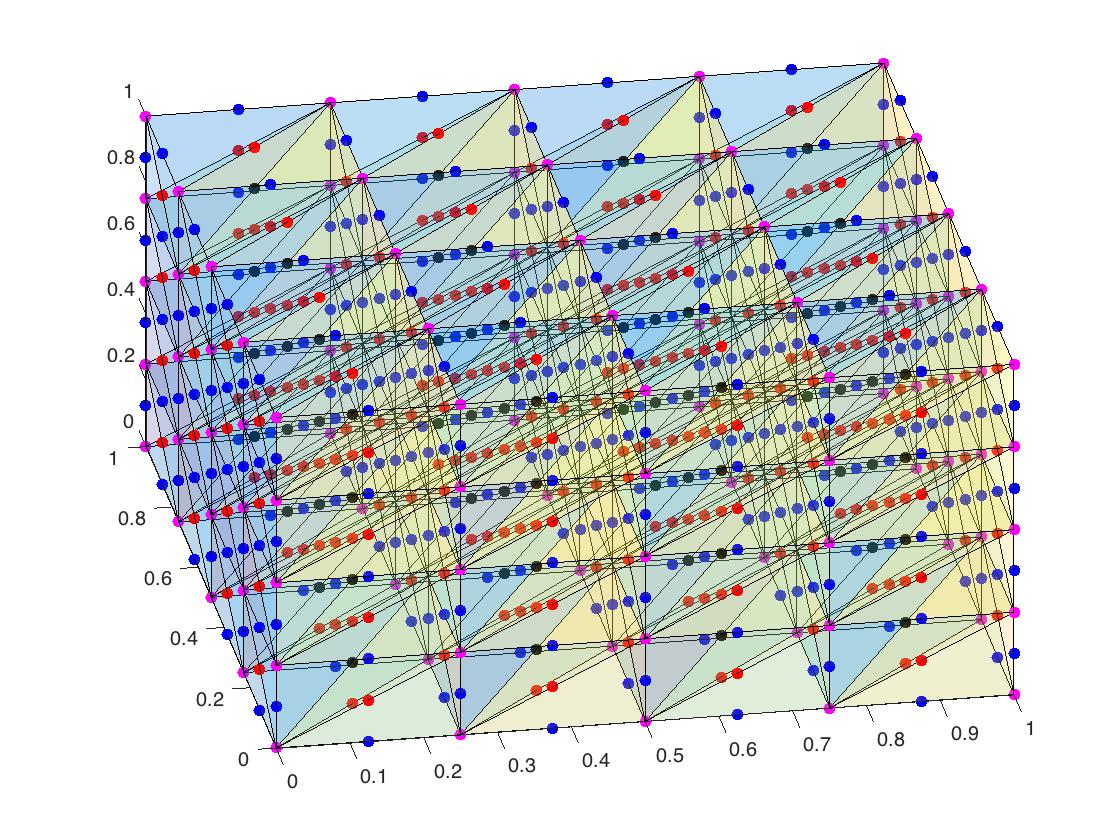}
		\caption{Refined complex $K_{c\cdot2^{-1}}$}
	\end{subfigure}
	\caption{The initial 3-dimensional primal triangulation and its first refinement are shown. A sequence of finer triangulations was created by gluing at each step $4$ smaller copies of the previous refinement into the initial cubic domain.}\label{cubic mesh}
\end{figure}

\section{Conclusion}

This paper fills a gap in the current literature, which has yet to offer any theoretical validation regarding the convergence of discrete exterior calculus approximations for PDE problems in general dimensions.
Namely, we prove that discrete exterior calculus approximations to the scalar Poisson problem converge at least linearly with respect to the mesh size
for quasi-uniform triangulations in arbitrary dimensions.
Nevertheless, it must be emphasized that this {\em first order} convergence result is only partially satisfactory.
In accordance with \parencite{Nong04}, the numerical experiments of Section \ref{s:numerics} display pointwise {\em second order} convergence of the discrete solutions over unstructured triangulations.
The same behavior was observed also for the discrete $L^2$ norm.
The challenges of
explaining the second order convergence, and of proving convergence for general $p$-forms persist.

Due emphasis must be given to the role played by compatibility in obtaining this result. The reliable framework of the continuous setting was reproduced at the discrete level to a sufficient extent so as to successfully provide a theorem of stability comparable to the one found in the classical PDEs literature. This makes another convincing case for the study and development of structure preserving discretization in general.

An accessory consequence of our investigations is the technical modification of the combinatorial definition for the boundary of a dual cell currently found in the early literature. We suggest that it is presently only compatible with the continuous theory up to a sign, and that accounting for the latter yields a new definition which agrees with the algorithms later described in \parencite{Desbrun2008}. A compatible extension of the usual discrete $L^2$-norm to $C^k(\dual K_h)$ was also explicitly introduced.

\section*{Acknowledgments}

This paper is the first author's first research article, and he would like to thank his family for its unwavering support on his journey to become a scientist.
The second author would like to thank
Gerard Awanou (UI Chicago) for his insights offered during the initial stage of this project, and
Alan Demlow (TAMU) for a discussion on interpreting some of the numerical experiments.
This work is supported by NSERC Discovery Grant, and NSERC Discovery Accelerator Supplement Program.
During the course of this work, the second author was also partially supported by the Mongolian Higher Education Reform Project L2766-MON.

\nocite{*}
\printbibliography

\end{document}